\documentclass[12pt,a4paper]{amsart}
\usepackage[top=40mm, bottom=40mm, left=35mm, right=35mm]{geometry}
\usepackage{mathptmx}
\usepackage{mathrsfs}
\usepackage{enumerate}
\usepackage{verbatim}
\usepackage{url}
\usepackage[all]{xy}
\usepackage{color}
\usepackage{hyperref}
\usepackage{enumerate}

\theoremstyle{plain}
\newtheorem{thm}{Theorem}[section]
\newtheorem{lem}[thm]{Lemma}
\newtheorem{prop}[thm]{Proposition}
\newtheorem{cor}[thm]{Corollary}
\newtheorem{prob}[]{Problem}
\theoremstyle{definition}
\newtheorem{de}[thm]{Definition}
\newtheorem{rem}[thm]{Remark}
\newtheorem{exam}[thm]{Example}

\newcommand{\Z}{\mathbb{Z}}
\newcommand{\modone}{\,\,(\textrm{mod}\,1)}
\newcommand{\R}{\mathbb{R}}
\newcommand{\N}{\mathbb{N}}
\newcommand{\X}{\mathbb{X}}
\newcommand{\T}{\mathbb{T}}

\newcommand{\U}{\mathcal{U}}
\newcommand{\V}{\mathcal{V}}
\newcommand{\W}{\mathcal{W}}
\newcommand {\ol}{\overline}
\newcommand{\ep}{\varepsilon}
\newcommand{\ra}{\rightarrow}

\newcommand{\id}{\mathrm{id}}
\DeclareMathOperator{\diam}{diam}

\DeclareMathOperator{\BD}{BD}
\DeclareMathOperator{\Aut}{Aut}

\numberwithin{equation}{section}

\begin{document}
\title{When all closed subsets are recurrent?}

\author{Jie Li}
\author{Piotr Oprocha}
\author{Xiangdong Ye}
\author{Ruifeng Zhang}
\address[J.~Li, X.~Ye]{Wu Wen-Tsun Key Laboratory of Mathematics, USTC, Chinese Academy of Sciences and
School of Mathematics, University of Science and Technology of China,
Hefei, Anhui, 230026, P.R. China}
\email{jiel0516@mail.ustc.edu.cn, yexd@ustc.edu.cn}
\address[P.~Oprocha] {AGH University of Science and Technology, Faculty of Applied Mathematics, al.~A. Mickiewicza 30, 30-059 Krak\'ow, Poland
	-- and -- 
	National Supercomputing Centre IT4Innovations, Division of the University of Ostrava,
	Institute for Research and Applications of Fuzzy Modeling,
	30. dubna 22, 70103 Ostrava,
	Czech Republic }
\email{oprocha@agh.edu.pl}
\address[R.~Zhang] {School of Mathematics, Hefei University of Technology, Hefei, Anhui, 230009, P.R. China}
\email{rfzhang@mail.ustc.edu.cn}

\begin{abstract}
In the paper we study relations of rigidity, equicontinuity and pointwise recurrence between
a t.d.s. $(X,T)$ and the t.d.s. $(K(X),T_K)$ induced on the hyperspace
$K(X)$ of all compact subsets of $X$, and provide some characterizations.

Among other examples, we construct a minimal, non-equicontinuous, distal and uniformly rigid t.d.s. and a t.d.s. which has dense small periodic sets but does not have dense distal points, solving that way open questions existing in the literature.
\end{abstract}

\date{\today}

\maketitle

\section{Introduction}

A \textit{topological dynamical system}, referred to more succinctly as just a t.d.s. is a pair $(X, T)$,
where $X$ is a compact metric space and $T\colon X \to X$ is a continuous surjective map from $X$ into
itself. A well known result of Birkhoff states that every t.d.s. has a recurrent point, i.e. there
are $x\in X$ and a sequence $n_i\ra\infty$ with $T^{n_i}x\ra x$. In \cite{KW} the question when
all points are recurrent was discussed by
Katznelson and Weiss, and transitive non-minimal systems
with all points are recurrent were constructed (see also \cite{AAB}). Strengthening the notion of pointwise recurrence,
Glasner and Maon \cite{GM89} introduced the notions of $n$-rigidity, weak
rigidity, rigidity and uniform rigidity, and showed  among other things that every rigid system
has zero topological entropy. Weiss in \cite{W95} proved that in fact positive $2$-rigidity
implies zero topological entropy. Note that this result is also followed by some conclusion
in \cite{BHR}, and it is an open question that if $2$-rigidity implies zero topological entropy when $(X,T)$ is a homeomorphism.

\medskip
Sensitivity is a key notion in the definition of Devaney's chaos.
When considering the system which is not sensitive, the notions of equicontinuous point and
almost equicontinuous system appear naturally. Similarly to an equicontinuous system, each almost
equicontinuous system is uniformly rigid,
see \cite{AAB, GW93}. A more striking result due to Hochman \cite{Hoch} is that each zero-entropy
ergodic measure preserving transformation is isomorphic to a positively $2$-rigid system.  Thus, from the point of view of measure theory, the class of rigid systems is very large, making then an important object of study in the framework of theory of dynamical systems.

\medskip
A t.d.s. $(X,T)$ induces in a natural way the system $(K(X),T_K)$ on the hyperspace of all compact sets (more details on particular hyperspaces can be found in \cite{Nad92}).
Bauer and Sigmund \cite{BS75} initiated a systematic study on the connections between
dynamical properties of $(X,T)$ and $(K(X),T_K)$. Particularly, they showed that $(X,T)$
is equicontinuous (resp. weakly mixing) if and only if so is $K(X)$, and provided an example which is distal but
$(K(X),T_K)$ is not distal.
Banks later in \cite{Ban05} showed that the transitivity of $(K(X),T_K)$ is coincident to its weak mixing property.
In \cite{LYY13} the authors further exploit these connections, and focus on periodic systems, $P$-systems,  $M$-systems, $E$-systems and disjointness.

\medskip
Following Bauer and Sigmund \cite{BS75} and Katznelson and Weiss \cite{KW} we study the question when
all closed subsets are recurrent. It turns out that $(K(X),T_K)$ is pointwise minimal if and only if $(X,T)$
is equicontinuous (see Theorem \ref{p.w. min.}), and that $K(K(X))$ is pointwise recurrent if and only
if $K(X)$ is weakly rigid if and only if $(X,T)$ is uniformly rigid (see Theorem \ref{uni. rigid} and
Corollary \ref{two-sided UR}). Specially, $K(X)$ is pointwise recurrent if and only if $(X,T)$ is
uniformly rigid whenever $X$ has a countable cardinality (see Theorem \ref{countable}). It is also shown
that the topological entropy of $(K(X),T_K)$ is zero when $K(X)$ is pointwise recurrent, i.e. 1-rigid (see Corollary \ref{zero ent.}).
Moreover, for a class of minimal distal systems and $n\in\N$ equivalent conditions when $K(X)$ is
$n$-rigid are given (see Corollary \ref{n-rigid}), and an example of a minimal distal non-equicontinuous
uniformly rigid t.d.s. is constructed (see Example \ref{ex2}).

Systems whose hyperspaces have dense recurrent points are considered in Section~ \ref{weak}. Among other things,
an example of a t.d.s. $(X,T)$ such that $(K(X),T_K)$ has dense distal points and $(X,T)$ does not have
the property is displayed (see Theorem~ \ref{K(X) P-system}). It answers a question in the positive left open in \cite{LYY13}.
We note that it was shown in \cite{LYY13} that a weakly mixing system $(X,T)$ satisfying
that $(K(X),T_K)$ has a dense set of distal points is disjoint from all minimal systems; and the fact
that a t.d.s. $(X,T)$ with a dense set of distal points is disjoint from all minimal systems was obtained previously
in \cite{O10, DSY12}.
%and Corollary \ref{openProb}

\medskip

We note that when writing the final version of the paper we found a preprint \cite{AAN}, where the authors
study the dynamical properties on the induced space $K(X)$. Among other things, they showed that (\cite[Theorem 5.4]{AAN})
$(X,T)$ is equicontinuous if and only if $K(X)$ is distal using a different method.
%, which is a weak conclusion of our Theorem \ref{p.w. min.}.

\section{Preliminaries}

Throughout this paper, denote by $\N$, $\Z_+$, $\Z$ and $\R$ the sets of positive integers,
nonnegative integers, integers and real numbers, respectively.

For a t.d.s. $(X, T)$ and $x\in X$, let $\textrm{orb}(x,T)=\{T^m x\colon m\in\Z_+\}$ be the
(positive) orbit of $x$. Fix $n\in \N$, write $(X^n, T^{(n)})$ as the $n$-fold product system
$(X\times X\times\dots\times X,T\times T\times\dots\times T)$, and set $\Delta_n=\{(x,x,\dots, x)\in X^n\colon x\in X\}$.

\subsection{Recurrence and its stronger forms}\label{sec:recurrencestronger}

Let $(X,T)$ be an invertible t.d.s. A point $x\in X$ is said to be \textit{positively recurrent}
(resp. \textit{negatively recurrent}) if there exists a sequence $n_i\ra +\infty$
(resp. $n_i\ra -\infty$) such that $T^{n_i}x \ra x$. Denote by $\textrm{Rec}(T)$
(resp. $\textrm{Rec}(T^{-1})$) the set of all positively recurrent (resp. negatively recurrent)
points. We say that  $(X,T)$ is \textit{pointwise positively recurrent}
(resp. \textit{pointwise negatively recurrent}) if $\textrm{Rec}(T)=X$ (resp.
$\textrm{Rec}(T^{-1})=X$). Note that there exists an example of a t.d.s. which
is pointwise positively recurrent but not pointwise negatively recurrent
(see \cite{AGW07}). We say $x\in X$ is \textit{recurrent}
if $x$ is either positively recurrent or negatively recurrent; and $(X,T)$ is
\textit{pointwise recurrent} if each point is recurrent.

A t.d.s. $(X, T)$ is \textit{(topologically) transitive} if for any two non-empty open
sets $U$ and $V$, the transfer time set $N(U,V)=\{n\in\Z_+\colon \ U\cap T^{-n} V\neq \emptyset \}$
is infinite; is \textit{weakly mixing} if $(X\times X,T\times T)$ is transitive; and is
\textit{mildly mixing} if $(X\times Y,T\times S)$ is transitive for any transitive t.d.s.
$(Y,S)$. We call $x\in X$ is \textit{a transitive point} if its orbit closure
$\ol{\textrm{orb}(x,T)}=X$. Let $\textrm{Trans}(T)$ be the set of transitive points.
It is well known that the orbit closure of a recurrent point is transitive.

A t.d.s. $(X, T)$ is \textit{minimal} if $\mathrm{Trans}(T)=X$. Call $x \in X$ a
\textit{minimal point} or \textit{almost periodic point} if the subsystem
$(\ol{\mathrm{orb}(x, T)}, T)$ is minimal. We say that $x \in X$ is a
{\it periodic point} if $T^n x=x$ for some $n \in \N$; and $(X,T)$ is
\textit{pointwise periodic} if each point in $X$ is periodic. The set of
all periodic points (resp. minimal points) of $(X, T)$ is denoted by
$\textrm{P}(T)$ (resp. $\textrm{AP}(T)$). It is easy to see that for
a transitive system we have $\textrm{P}(T)\subset \textrm{AP}(T)\subset
\mathrm{Trans}(T)\subset \textrm{Rec}(T)$.

A pair $(x, y) \in X^2$ is called \textit{proximal} if there is a sequence
$n_i\ra +\infty$ such that $d(T^{n_i} x, T^{n_i} y)\ra 0$; and  \textit{regionally proximal}
if for each $\ep>0$ there are $x',y'\in X$ and $k\in \N$ with $d(x,x')<\ep,\ d(y,y')<\ep$ and
$d(T^k x', T^k y') <\ep$.  The subset consisting of all proximal (resp. regionally proximal)
pairs is denoted by $P(X, T)$ (resp. $Q(X, T)$). $x\in X$ is said to be \textit{a distal point}
if $x$ is only proximal to itself in its orbit closure; and be \textit{an equicontinuous point}
if for each $\ep>0$ there is $\delta>0$ such that $\diam T^n (B(x, \delta)) < \ep$ for each $n\geq 0$,
where $B(x, \delta)$ is the open ball centered at $x$ with radius $\delta$. Now we can say a
t.d.s. $(X,T)$ is \textit{distal} if all points are distal, or if $P(X,T)=\Delta_2$; and is
\textit{equicontinuous} if all points are equicontinuous.

Since $T$ is surjective, it is not hard to see that when $(X,T)$ is equicontinuous it is distal.
%pointwise positively recurrent and consequently it does not contain proximal pairs.
%It immediately implies that every equicontinuous system is an invertible distal system. Fix any $\ep>0$ and $x\in X$.
It is also not hard to see that if $T$ is equicontinuous then $Q(X,T^{-1})=\Delta_2$.
It was first proved by Veech \cite{V68}
%Auslander (see \cite{A88})
that maximal equicontinuius factor is induced by the smallest closed equivalence relation containing $Q(X,T)$,
in particular $Q(X,T^{-1})=\Delta_2$ implies that $T^{-1}$ is equicontinuous.

\medskip
Another direction to strengthen the recurrence are various notions of rigidity (e.g. see \cite{GM89}).
%{\color{red}PO: n-rigid is not present in Glasner-Maon paper; can we provide a reference where it appeared first?}
Let $n\in\N$. An invertible t.d.s. $(X,T)$ is called \textit{$n$-rigid} if each $n$-tuple $(x_1,\dots,x_n)\in
(X^n,T^{(n)})$ is a recurrent point; \textit{weakly rigid} if $(X,T)$ is $n$-rigid for each $n\in\N$;
\textit{rigid} if there is $m_i\ra\infty$ such that $T^{m_i} \to \id$ pointwise, where $\id$ is
the identity map; and \textit{uniformly rigid} if there is $m_i\ra\infty$ such that $T^{m_i} \to \id$
uniformly on $X$. It is equivalent to say that for each $\ep>0$ there is $m\in\Z$ such that $d(T^m x,x)<\ep$
for each $x\in X$.
The same way  (when the map is not necessarily invertible) we can define \textit{positively $n$-rigid}
and \textit{positively weakly rigid} systems, replacing
recurrence by the positive recurrence in the definition. Also when defining rigid and uniformly rigid,
we can drop the assumption that $(X,T)$ is invertible.
%{\color{red} Here I am not sure that positive rigidity is equivalent to negative rigidity.}

It is known that a minimal equicontinuous system is uniformly rigid and there are minimal weakly mixing
uniformly rigid systems \cite{GM89}. A transitive system with an equicontinuous, transitive point is
called \textit{almost equicontinuous}, and such systems  are uniformly rigid which may be proximal,
see \cite{GW93}. A minimal rigid but not uniformly rigid system is constructed in \cite{K}.
A minimal distal system is weakly rigid, and the system $(X,T)$ defined by $T(x,y)=(x+\alpha,x+y)$
on $\T^2$ is not rigid, see \cite{GM89}.
%An example of different natural is constructed in this paper.

It is clear that a minimal system is $1$-rigid, and it is easy to see that the Denjoy minimal system on the circle
is $1$-rigid but not $2$-rigid. It is an open question if for $n\ge 2$ there is a system which is $n$-rigid but not $n+1$-rigid,
though the general opinion is that such examples should exist.

\subsection{Factor and extension}

Let $(X, T)$ and $(Y,S)$ be two systems and $\pi\colon X\to Y$. We say $X$ is an \textit{extension}
of $Y$ or $Y$ is a \textit{factor} of $X$ if $\pi$ is continuous onto and interwines the actions,
i.e. $\pi\circ T=S\circ \pi$. In this case call $\pi$ a \textit{factor map}.

Let $(X,T)$ be a t.d.s. A self homeomorphism $\xi$ of $X$ is an automorphism of $(X,T)$ if it
commutes with $T$, i.e. $\xi \circ T=T\circ \xi$. We let $\Aut(X,T)$ be the collection of all
automorphisms of $(X,T)$. If $K$ is a compact subgroup of $\Aut(X,T)$, then the map $x\mapsto
Kx$ defines a factor map $\pi\colon(X,T)\to (Z,\widetilde{T})$ with $Z=X/K$ and $R_\pi
=\{(x,kx)\colon x\in X, k\in K\}$. Such an extension is called a \textit{group extension}.
A special group extension will be considered in this paper, and we refer it to be
\textit{skew product} \cite{Fur81}, i.e. for some t.d.s. $(Y,S)$ and compact group $G$,
form $X=Y\times G$ with $T(y,g)=(S y, \phi(y) g)$, where $\phi\colon Y\to G$ is a continuous map.

\subsection{Subset of integers}

Let $S$ be a subset of $\Z$ (resp. $\Z_+$). We say that $S$ is \textit{syndetic} (resp.
\textit{positively syndetic}) if it has a bounded gap, i.e. there is $N\in \N$
with $\{i,i+1,\dots, i+N\}\cap S\neq\emptyset$ for all $i\in \Z$ (resp. $i\in \Z_+$).

Let $\{p_i\}_{i=1}^\infty$ be an infinite sequence in $\N$. Set
$$FS(\{p_i\}_{i=1}^{\infty})=\{p_{i_1}+p_{i_2}+\dots+p_{i_n}\colon 1 \le i_1<i_2<\dots<i_n,\ n \in \N\}.$$
A subset $A\subset \N$ is said to be an \textit{$IP$-set} if it contains some
$FS(\{p_i\}_{i=1}^{\infty})$, and to be an \textit{$IP^*$-set} if it has non-empty intersection with any $IP$-sets.

\medskip
Let $(X, T)$ be a t.d.s. For a point $x \in X$ and open subsets $U, V \subset X$, we put the transfer times sets:
\begin{align*}
  N_T(x, U)& =\{n \in \Z_+\colon T^n x \in U\},\ \mbox{and} \\
   N_T(U,V) & =\{n\in \Z_+\colon U\cap T^{-n} V\neq\emptyset\}.
\end{align*}
When the acting map $T$ is clear from the context, we simply write $N(x,U)$ and $N(U,V)$.

The following characterizations are well known, see for instance \cite{Fur81}.
\begin{lem}
Let $(X, T)$ be a t.d.s. If $x\in X$ and $U$ is any neighborhood of $x$, then
\begin{enumerate}

  \item $x$ is positively recurrent if and only if $N(x, U)$ is an $IP$-set;

  \item $x$ is minimal if and only if $N(x, U)$ is a positively syndetic set;

  \item $x$ is distal if and only if $N(x, U)$ is an $IP^*$-set.

\end{enumerate}
\end{lem}

When related to the weak mixing property, we have the following results
(see \cite[Lemma 5.1]{HY05} and \cite[Theorem 9.12]{Fur81}, respectively):

\begin{lem}\label{w.m.}
  Let $(X, T)$ be a t.d.s. with infinite cardinality. Then
  \begin{enumerate}

    \item $X$ is weakly mixing if and only if $X$ is transitive and for any non-empty open subset $U\subset X$, there is $n\in \Z_+$ such that $n, n+1\in N(U,U)$ ;

    \item if $(X, T)$ is minimal and weakly mixing, then no point of $X$ is distal.

  \end{enumerate}
\end{lem}

\medskip
Let $J$ be a subset of $\Z_+$. The \textit{density} and \textit{upper Banach density} of $J$ are defined by
$$d(J)=\lim_{n \ra \infty} \frac{\# \{ J \cap [0, n-1]\}}{n} \mbox{\ \ and \ }  \BD^*(J)=\limsup_{N-M \ra \infty} \frac{\# \{J \cap [M,N] \} }{N-M+1} .$$
where $I$ is over all non-empty finite intervals of
$\Z_+$ and $\#\{\cdot\}$ denotes the cardinality of the set.

The following lemma is also well known (see \cite{Fur81} or \cite[Propostion 2.3]{YZ08}, etc), and we omit the simple proof.

\begin{lem}\label{pubd-ip}
If $J$ has positive upper Banach density and $Q$ is an $IP$-set,
then there exists $l\in Q$ such that
$$\BD^*(J\cap(J-l))> 0.$$
\end{lem}

\subsection{Hyperspace}

Let $X$ be a compact and metrizable space with metric $d$. Define $K(X)=\{A\subset X \colon A\neq\emptyset, \,\, \ol{A}=A\}$, that is,
the collection of non-empty closed subsets of $X$. Endow a \textit{Hausdorff metric} $d_H$ on $K(X)$ defined by
$$d_H(A, B)=\max \left\{\max_{x \in A} \min_{y \in B} d(x, y),
\max_{y \in B} \min_{x \in A} d(x, y)\right\}$$
for $A, B \in K(X)$, or equivalently by
$$d_H(A, B)=\inf \{\ep>0 \colon B_\ep(A) \supset B,\, B_\ep(B) \supset A \},$$
where $B_\ep(A)=\{x\in X\colon d(x, a)<\ep\;\mbox{for some }\; a\in A\}$ is an $\ep$-neighborhood of $A$ in $X$.
We call the space $K(X)$ with the topology induced by $d_H$, which is the \textit{Vietoris topology}
(see \cite[Theorem 4.5]{Nad92}), as hyperspace of $X$. Note that this topology turns $K(X)$ into a compact space.

\medskip
Fix $n\in \N$, denote $K_n(X)=\{A\in K(X)\colon |A|\le n\}$. It is easy to see that $K_n(X)$ is closed
and $\cup_{n\ge 1} K_n(X)$ is dense in $K(X)$ (see \cite[Lemma 2]{BS75}).

\medskip
For any non-empty open subsets $U_1, \dots, U_n$ of $X$, $n \in \N$, let
$$\langle U_1, \dots, U_n\rangle=\{A \in K(X)\colon  A \subset
\cup_{i=1}^n U_i \mbox{ and } A \cap U_i \neq \emptyset \mbox{ for each } i=1, \dots, n\}.$$
We can check that $\langle U_1, \dots, U_n\rangle$ is a non-empty open subset of $K(X)$. Moreover the following family
$$\{\langle U_1, \dots, U_n\rangle\colon U_1, \dots, U_n \mbox{ are non-empty open subsets of } X, n \in \N\}$$
forms a basis for the Vietoris topology \cite{Nad92}.

\medskip
Now let $(X, T)$ and $(Y,S)$ be two systems and $\pi\colon X\to Y$ be a factor map. Define the induced map
$\pi_K\colon K(X) \ra K(Y)$ by
$$\pi_K(A)=\pi (A) \mbox{ for } A \in K(X).$$
It is easy to verify that $\pi_K$ is also a factor map. Particularly when $Y=X$ we obtain an induced continuous
surjective transformation $T_K$, such that $(K(X), T_K)$ is a t.d.s.

\medskip
Recall that $(X, T)$ is a $P$-system if it is transitive with dense periodic points; and has
\textit{dense small periodic sets} \cite{HY05} if for any non-empty open subset $U\subset X$, there
exists a closed subset $A$ of $U$ and $n \in \N$ such that $T^n A \subset A$. Now we present some results
we will use in the sequel. See \cite[Theorem 2]{Ban05} and \cite[Theorem 1.1]{Li2014} respectively for details.
\begin{lem}\label{P-system}
Let $(X,T)$ be a t.d.s. and $(K(X),T_K)$ be the hyperspace. Then
\begin{enumerate}

\item $(X,T)$ is weakly mixing if and only if $(K(X),T_K)$ is transitive;

\item $(X,T)$ is weakly mixing and has dense small periodic sets if and only if $(K(X),T_K)$ is a $P$-system .

\end{enumerate}
\end{lem}

\section{Pointwise minimality}

A t.d.s. $(X, T)$ is \textit{pointwise minimal} (resp. \textit{distal}) if all points in $X$ are minimal
(resp. distal). Note that if $(K(X),T_K)$ is pointwise minimal (resp. distal), then so is $(X,T)$, since it is a subsystem of $(K(X),T_K)$.
But there are many examples showing the reciprocal is not valid (see \cite[theorem 3]{BS75}, etc.).
Now we provide a simpler example.

\begin{exam}\label{ex1}
Let $\T^1=\R/ \Z$. Consider the map $T\colon\T^2\to \T^2, (x,y)\mapsto (x+\alpha,x+y\modone)$,
where $\alpha$ is an irrational number. Then for any $n\in \N$,
$$
T^n(x,y)=(x+n\alpha,y+nx+a(n)\alpha \modone)
$$
with $a(n)=n(n-1)/2$. Note that $\T^2$ is distal, since it is a distal extension of the distal
system $\T^1$. Now we show $(K(\T^2),T_K)$ is not pointwise positively recurrent, and a fortiori neither distal nor pointwise minimal.

Let $y_0\in \T^1$ and $A=\{(x,y_0)\colon x\in\T^1\}\in K(\T^2)$. Fix any $n\in\N$ and note that
$$
d_H(T_K^n A,A)=d_H(\{(x+n\alpha, y_0+nx+a(n)\alpha) \colon x\in \T^1\},\T^1\times \{y_0\}).
$$
Put $\ep_0=1/10$ and choose $x_n\in \T^1$ such that
$$
\ep_0 \le nx_n+a(n)\alpha \modone\le 1-\ep_0.
$$
 Thus for any $(x,y_0)\in A$,
by the definition of the Hausdorff metric we get
\begin{equation*}
  \begin{split}
     d_H(T_K^n A,A) & \ge \min_{x\in \T^1} d((x_n+n\alpha, y_0+nx_n+a(n)\alpha),(x,y_0)) \\
       & \ge d(y_0+nx_n+a(n)\alpha, y_0)\ge \ep_0
   \end{split}
\end{equation*}
This implies that $N(A, B_{\ep_0}(A))\setminus \{0\}=\emptyset$ and so $A$ is not positively recurrent.
\end{exam}

A natural question is when does the pointwise minimality (resp. distality) of $(K(X),T_K)$ hold?
The above Example \ref{ex1} suggests that $K(X)$ is pointwise minimal (resp. distal) if and only if
it is equicontinuous in some sense. To give a strict proof, we need the notion of {\it locally almost periodic system},
which was first introduced by Gottschalk in \cite{Got56}. Here what we need is a particular case,
only considering actions of group $\Z$. So we reformulate the definition to the following form.

\begin{de}
Let $(X,T)$ be a t.d.s. A point $x\in X$ is \textit{locally almost periodic}, if for every neighborhood
$U$ of $x$, there exists a neighborhood $V$ of $x$ and a syndetic set $F\subset \Z$ such
that $T^n V\subset U$ for all $n\in F$. $(X,T)$ is called a \textit{locally almost periodic system}
if every point of $X$ is locally almost periodic.
\end{de}

Next we need to invoke an useful lemma on the characterization of equicontinuity, and for completeness we provide a direct proof. A proof can be also derived from results in \cite{Got56}.

\begin{lem}\label{equi.}
$(X,T)$ is equicontinuous if and only if $(X,T)$ is distal and locally almost periodic.
\end{lem}

\begin{proof}
As we mentioned at the beginning in Section~\ref{sec:recurrencestronger}, it is well known that every equicontinuous t.d.s. $(X,T)$
is distal and invertible. Hence it remains to show that it is locally almost periodic.
Fix any $x\in X$ and $\ep>0$. Since both $(X,T)$ and $(X,T^{-1})$ are equicontinuous,
there is $\delta=\delta(\ep)>0$ such that for any $y\in B(x,\delta)$ we have
$d(T^n x,T^n y)<\ep/2$ for all $n\in \Z$. But $x$ is distal, hence $x$ is minimal and so we can find a
syndetic set $F\subset \Z$, such that $d(x,T^n x)<\ep/2$ for each $n\in F$. It immediately implies that  $T^n B(x,\delta)\subset B(x,\ep)$ for all $n\in F$. Indeed, $(X,T)$ is locally almost periodic.

Now assume that $(X,T)$ is distal and locally almost periodic. Let $(x_1,x_2)\in X^2\setminus\Delta_2$. We claim that there exist $\delta>0$ and two positive real numbers $r_1,r_2$ such that $d(T^n x, T^n y)>\delta$
for any $(x,y)\in B(x_1,r_1)\times B(x_2,r_2)$ and $n\in \Z$.
%
%Now we check the claim. First we show that there exist $\delta'>0$ and $r_1>0$ such that $\inf_{n\in \Z}d(T^n x, T^n x_2)\geq \delta'$ for each $x\in B(x_1, r_1)$.

By the distality of $(X,T)$ and $(X,T^{-1})$, there is $\ep_0>0$ such that $d(T^n x_1, T^n x_2)\geq\ep_0$ for all $n\in \Z$. By assumptions $x_1$ is locally almost periodic, hence we have a syndetic set $F\subset \Z$ and $r_1>0$ such that $T^p B(x_1, r_1)\subset B(x_1,\ep_0/3)$ for all $p\in F$. Take $m_1\in \N$ and $\delta'>0$ such that $F+[0,m_1]\supset \Z$ and $d(T^j x, T^j y)<\ep_0/3$ for $j=0,1,\dots,m_1$, provided that $d(x,y)<\delta'$. Suppose that there are $x\in B(x_1, r_1)$ and $n_1\in \Z$ such that
$d(T^{n_1} x, T^{n_1} x_2)<\delta'$. There is $j\in [0,m_1]$ such that $n_1+j\in F$. Denote $p_{_0}=n_1+j$ and observe that $d(T^{p_{_0}} x,T^{p_{_0}} x_2)<\ep_0/3$. Since $T^{p_{_0}} B(x_1, r_1)\subset B(x_1,\ep_0/3)$ and $x\in B(x_1, r_1)$, we have
$$d(T^{p_{_0}} x_1,T^{p_{_0}} x_2)\le d(T^{p_{_0}} x_1, x_1)+d(x_1,T^{p_{_0}}x)+d(T^{p_{_0}}x,T^{p_{_0}}x_2)<\ep_0,$$
which is a contradiction. This proves that there are $\delta'>0$ and $r_1>0$ such that $\inf_{n\in \Z}d(T^n x, T^n x_2)\geq \delta'$ for each $x\in B(x_1, r_1)$.

Now by the local almost periodicity of $x_2$, choose $r_2>0$ such that for some syndetic subset $F_2\subset \Z$ we have $T^qB(x_2,r_2)\subset B(x_2,\delta'/4)$ for every $q\in F_2$. Take any $m_2>m_1\in\Z$ with $F_2+[0, m_2]\supset \Z$ and let $\delta>0$ be such that if $d(x,y)<\delta$
then $d(T^k x,T^k y)<\delta'/2$ for $k=0,1,\dots, m_2$. If there are $(x,y)\in B(x_1,r_1)\times B(x_2,r_2)$ and $n_2$ such that $d(T^{n_2} x,T^{n_2} y)<\delta$, then there is $0\le k\le m_2$ with $q_{_0}=n_2+k\in F_2$ such that $d(T^{q_{_0}}x,T^{q_{_0}}y)<\delta'/2.$
But then
$$d(T^{q_{_0}}x,T^{q_{_0}}x_2)\le d(T^{q_{_0}}x,T^{q_{_0}}y)+d(T^{q_{_0}} x_2,T^{q_{_0}}y)<\delta',$$
which again is a contradiction. This implies that $d(T^n x, T^n y)>\delta$
for every $(x,y)\in B(x_1,r_1)\times B(x_2,r_2)$ and $n\in \Z$.
Indeed the claim holds, which implies that $(x_1,x_2)\notin Q(X,T)$.
We have just proved that $Q(X,T)=\Delta_2$ and so $(X,T)$ is equicontinuous.
\end{proof}

Now we are ready to show the following.

\begin{thm}\label{p.w. min.}
Let $(X,T)$ be a t.d.s. The following statements are equivalent:
\begin{enumerate}
\item\label{p.w.min.:1} $(X,T)$ is equicontinuous,
\item\label{p.w.min.:2} $(K(X),T_K)$ is equicontinuous,
\item\label{p.w.min.:3} $(K(X),T_K)$ is distal,
\item\label{p.w.min.:4} $(K(X),T_K)$ is pointwise minimal.
\end{enumerate}
\end{thm}

\begin{proof}
The equivalence of \eqref{p.w.min.:1} and \eqref{p.w.min.:2} can be found in \cite[Proposition 7]{BS75}. Trivially we have $\eqref{p.w.min.:2} \Longrightarrow
\eqref{p.w.min.:3} \Longrightarrow \eqref{p.w.min.:4}$, so it remains to show that \eqref{p.w.min.:4} implies \eqref{p.w.min.:1}.

Assume $(K(X),T_K)$ is pointwise minimal. First, we claim that $(X,T)$ is distal. Let if possible there
exists a pair $(x,y)\in P(X,T)$. Then there are $z\in X$ and $\{n_i\}_{i=1}^\infty$ such that
$\lim_{i\ra \infty} T^{n_i}x=\lim_{i\ra \infty} T^{n_i}y=z$. Note that $\{z\}$, $\{x,y\}$ are
minimal points in $K(X)$. Hence $\{x,y\}\in \ol{\textrm{orb}(\{z\},T_K)}$, which implies that $x=y$, so the claim holds.
Therefore, $X$ is invertible and $(X,T^{-1})$ is also distal.

Now we aim to show $(X,T)$ is locally almost periodic. Fix any $x\in X$ and any open neighborhood $U$ of $x$. Take any open set $V$ such that
$x\in V\subset \ol{V}\subset U$. Then $\ol{V}\in \langle U\rangle$ and since $(K(X),T_K)$ is
pointwise minimal, there is a syndetic set $F\subset \Z$ such that $T_K^n \ol{V}\in
\langle U\rangle$ for any $n\in F$. Equivalently it means that $T^n V\subset T^n \ol{V}\subset U$ for all $n\in F$, showing  that
$x$ is a locally almost periodic point, which by Lemma~\ref{equi.} implies that $(X,T)$ is equicontinuous, that is \eqref{p.w.min.:1} holds. The proof is completed.
\end{proof}

\begin{rem}
The notions of pointwise minimality, distality and equicontinuity are different in general.
For example, the Denjoy extension of the irrational rotation (restricted to the nonwandering set which is a Cantor set; see \cite{Dev}) is pointwise minimal but not distal. A disk rotated at
different rates around a common center or the Example \ref{ex1} is distal but not
equicontiunous. Theorem \ref{p.w. min.} shows these three properties are the same for map induced on the hyperspace.
\end{rem}

\section{Levels of rigidity}

As mentioned before, the concepts of weak rigidity, rigidity and uniform rigidity were first
introduced by Glasner and Maon \cite{GM89}. It is known that they keep strict inclusion relationship
in general, but for a minimal distal system, rigidity is equivalent to uniform rigidity, and for a
minimal zero-dimensional system, weak rigidity is identical with equicontinuity \cite{GM89}. A result due to Dong \cite{Dong} shows that if a minimal nilsystem is rigid then it is equicontinuous.
Here we point out that group extension of a rigid system is weakly rigid, which can yield from a stronger result below. Note that in the following proof we strongly rely on the theory of Ellis semigroups. The reader not familiar with this topic is referred to \cite{A88,AAG}.

We will denote by $E(X,T)$ the enveloping (or Ellis) semigroup  associated with t.d.s. $(X,T)$, that is
the compact semigroups of $X^X$ defined as the closure of $\{T^n : n\in\Z\}$ in $X^X$.

\begin{thm}
Let $(Y,S)$ be a weakly rigid t.d.s. and $(X,T)$ is a distal extension of $(Y,S)$. Then $(X,T)$ is weakly rigid.
\end{thm}

\begin{proof}
Fix any integer  $n\geq 1$ and
for each $y\in Y^n$ put $F_y=\{p\in E(Y^n,S^{(n)}): py=y\}$.
Since $Y^Y$ is endowed with the topology of pointwise convergence and $(Y,S)$ is weakly rigid, we see that $\id^{(n)} \in E(Y^n,S^{(n)})$.
Clearly $\id^{(n)} \in F_y$ for every $y\in Y^n$, hence $F=\cap_{y\in Y^n}F_y$ is a non-empty compact semigroup of $E(Y^n,S^{(n)})$. Then we can apply Ellis-Namakura Lemma (see \cite[Lemma~2.1]{AAG}) to $F$
obtaining an idempotent element $u\in F$.

Let $\pi\colon X\to Y$ be the distal extension, and then define $\pi_n$ the natural factor map given by $\pi_n(x_1,\dots,x_n)=(\pi(x_1),\ldots,\pi(x_n))$.
By \cite[Theorem 7]{A88} there exists a unique continuous semigroup homomorphism $\theta\colon E(X^n,T^{(n)})\to E(Y^n,S^{(n)})$ such that $\pi_n(xq)=\theta(q)\pi_n(x)$
for every $x\in X^n$ and $q\in E(X^n,T^{(n)})$.
Clearly $J=\theta^{-1}(u)$ is a closed semigroup in $E(X^n,T^{(n)})$ so again by Ellis-Namakura Lemma there is an idempotent
$v\in J$. Fix any $x\in X^n$ and observe that $\pi_n(xv)=u\pi_n(x)=\pi_n(x)$, which shows that $x$ and $vx$ are in the same fiber of $\pi_n$.  By \cite[Proposition~2.4]{AAG} we immediately obtain that $(x,vx)$ is a proximal pair,
hence for each coordinates $i$ we see that $(x_i, (vx)_i)$ is a proximal pair. But $\pi(x_i)=\pi((vx)_i)$ hence $x_i=(vx)_i$ because
$\pi$ is distal. This shows that $x=vx$ and so again by \cite[Proposition~2.4]{AAG} we see that $x$ is a recurrent point of $T^{(n)}$.
This proves that $T$ is $n$-rigid for every $n$, completing the proof.
%Note that the factor map $\pi$ is distal.
%{\color{red}PO: Any comment/reference explaining this distality?}
%Therefore, since $x$ and $vx$ are proximal, we must have $x=vx$. This shows that $x$ is positively recurrent.
%Note that $$
%
%This also implies that for any $x_1,\ldots,x_n\in X$, $u(x_1,\dots, x_n)=(x_1,\dots, x_n)$,
%which yields the weak rigidity of $X$.
%%Moreover, since all the forms like $\mathcal{U}(x_1,\dots,x_k; U_1,\dots, U_k)=\{f\in
%%E\colon f(x_1)\in U_1,\dots, f(x_k)\in U_k\}$ (here $x_1,\dots, x_k\in X$ and $U_1,\dots, U_k$ are non-empty
%%open subsets of $X$) induces the pruduct topology on $E$, then for any $n\in \N$, we can easily obtain that
%%For $x\in X$, then $A_x=\{q\in E\colon qx=x\}$ is non-empty open compact semigroup, and it has finite intersection property, i.e. $\bigcap_{x\in X}A_x\neq\emptyset$. Hence there is $s\in bigcap_{x\in X}A_x$ such that $sx=x$ for all $x\in X$, which is equivalent to weak rigidity.
\end{proof}

In the sequel we will investigate the relations between various types of rigidity for $(X,T)$ and
$(K(X),T_K)$. It is easy to see that $(K(X),T_K)$ is weakly rigid, so is its subsystem $(X,T)$. But the converse implication is not necessarily true. For example,  set $\T^1\times \{y\}$ in t.d.s. $(X,T)$ in Example~\ref{ex1} is not a positively
recurrent point in $(K(X),T_K)$, but $(\T^2,T)$ is minimal and distal, therefore it is weakly rigid by
\cite[Corollary 6.2]{GM89}.

Similarly the rigidity of $(K(X),T_K)$ implies the same for $(X,T)$, but converse does  not
always hold. Rigid t.d.s. which is not uniformly rigid constructed in \cite{GM89} can serve as an example. In this example we
take $X=\{re^{i\theta}\colon 0\le \theta\le 2\pi, r=1-2^{-n},n=1,2,3,\dots\: \mbox{or} \: r=1\}$
and define: $Tz=z\exp(2\pi i\cdot2^{-n})$ when $|z|=1-2^{-n}$ and $Tz=z$ if $|z|=1$. This map is rigid with respect to sequence $n_k=2^k$, but
the set $R=\{r\colon\,r=1-2^{-n},n=1,2,3,\dots\: \mbox{or} \: r=1\}$ is not
a positively recurrent point of the hyperspace.

But the situation changes when consider the property of uniform rigidity. It turns out
that uniform rigidity holds always for both $(X,T)$ and $(K(X),T_K)$, and furthermore on
$(K(X),T_K)$ all the properties of rigidity considered so far coincide. Strictly speaking, we have the following.

\begin{thm}\label{uni. rigid}
Let $(X,T)$ be a t.d.s. Then the following statements are equivalent:
\begin{enumerate}
\item\label{uni. rigid:1} $(X,T)$ is uniformly rigid;
\item\label{uni. rigid:2} $(K(X),T_K)$ is uniformly rigid;
\item\label{uni. rigid:3} $(K(X),T_K)$ is rigid;
\item\label{uni. rigid:4} $(K(X),T_K)$ is weakly rigid.
\end{enumerate}
\end{thm}

\begin{proof}
First, we show that $\eqref{uni. rigid:1} \Longrightarrow \eqref{uni. rigid:2}$. Assume that $(X,T)$ is uniformly rigid. That is, given $\ep>0$
there is $n\in \N$ such that $d(T^n x,x)<\ep$ for all $x\in X$. Then for each $A\in K(X)$ we have
$d_H(A, T_K^n A)\leq \max_{x\in A} d(T^n x,x)<\ep$. This implies that $(K(X),T_K)$ is uniformly rigid.

Implications $\eqref{uni. rigid:2} \Longrightarrow \eqref{uni. rigid:3}\Longrightarrow \eqref{uni. rigid:4}$ follow by definition, hence it remains to show $\eqref{uni. rigid:4} \Longrightarrow \eqref{uni. rigid:1}$.
Assume that $(K(X),T_K)$ is weakly rigid and fix any $\ep>0$. Let points $x_1,\ldots,x_n\in X$ be such that $X=\bigcup_{i=1}^nB(x_i,\ep/4)$. Denote $A_i=\ol{B(x_i,\ep/4)}\in K(X)$
and observe that since $(K(X),T_K)$ is $n$-rigid for each $n\in\N$, it follows that $(A_1,\dots,A_n)$ is a
recurrent point of $(K(X)^n,T_K^{(n)})$. Thus, there is $j\in\Z$ such that
$$(T_K^{(n)})^j(A_1,\dots,A_n)\in \langle B(x_1,\ep/2)\rangle\times \dots\times \langle B(x_n,\ep/2)\rangle.$$
Therefore $T^j_K A_i\in \langle B(x_i,\ep/2)\rangle$ which equivalently means that $T^j A_i\subset B(x_i,\ep/2)$, where $i=1,\dots,n$. For every $x\in X$ we can find $i$ such that $x\in A_i$ and so $d(T^j x,x)<\ep$ for each $x\in X$. This shows that $(X,T)$
is uniformly rigid, completing the proof.
\end{proof}

\section{Pointwise recurrence}

%\subsection{General discussion}

In this section we focus on the pointwise recurrence on hyperspace. Firstly we prove a few general facts on pointwise recurrence.

\begin{prop}\label{fact}
Let $(X,T)$ be a t.d.s. and $(K(X),T_K)$ be pointwise positively recurrent. Then we have:
\begin{enumerate}
\item\label{fact:1} Every non-trivial minimal subsystem of $(X,T)$ is not mildly mixing. Particularly, if $(X,T)$ is minimal and $\#X>1$ then $(X,T)$ is not mildly mixing;
\item\label{fact:2} $(X,T)$ is positively weakly rigid.

\end{enumerate}
\end{prop}

\begin{proof}
First we prove \eqref{fact:1}. By \cite{HY04} if $(Y,T)$ is minimal and mildly mixing then for every pair of non-empty open
subsets $U$ and $V$, the set of transfer times $N(U,V)$ is an $IP^*$-set. Let $(Y,T)$ be a mildly mixing minimal subsystem
of $(X,T)$ and $\#Y>1$. Then we can find open sets $U,V$ intersecting $Y$ and such that $\ol{U}\cap \ol{V}=\emptyset$.
Take any open set $W$ such that $W\subset \ol{W}\subset U$ and $W\cap Y\neq \emptyset$. Denote $A=\ol{W}\cap Y\in K(X)$.
Since $(K(X),T_K)$ is pointwise positively recurrent, $A$ is positively recurrent and so $N(A,\langle U \rangle)$ is an $IP$-set.
But $A$ has non-empty interior in $Y$ and so $N(A,V)$ is an $IP^*$-set. In particular,
$N(A,V)\cap N(A,\langle U \rangle)\neq \emptyset$, which implies that $T^jA\cap V\neq \emptyset$ and $T^j A \subset U$
for some $j\in \N$. This is a contradiction, because  $U\cap V=\emptyset$.

Now, let us proceed with the proof of \eqref{fact:2}.
It is known that if $\pi\colon Y\to Z$ is a factor map and $z\in Z$ is positively recurrent then
there is a positively recurrent point $y\in Y$ such that $\pi(y)=z$. Now assume that $(K(X),T_K)$ is
pointwise positively recurrent. Then this property is shared by each subsystem $(K_n(X),T_K)$, where $n=1,2,\ldots$. Since $(X^n, T^{(n)})$ is
an extension of $(K_n(X), T_K)$, for each
$(x_1,\dots,x_n)\in X^n$ we obtain some permutation $(i_1,\dots, i_n)$ of $(1,\dots,n)$ such that $(x_{i_1},\dots,x_{i_n})$ is positively recurrent. But dynamics of $\{x_1,\dots,x_n\}$ under $T^n$ is exactly the same as that of $(x_{i_1},\dots,x_{i_n})$ which shows that
$\{x_1,\dots,x_n\}$ is positively recurrent in $(X^n, T^{(n)})$. We obtain that $(X^n, T^{(n)})$ is pointwise positively
recurrent, which shows that  $(X,T)$ is positively $n$-rigid for each $n\in\N$. The proof is finished.
\end{proof}

\begin{rem}\label{rem}
\begin{enumerate}
\item\label{rem:1} In \cite[Theorem 4.3]{HY04} the authors showed that a non-trivial uniformly rigid system cannot be mildly mixing. So comparing with Proposition~ \ref{fact}(1), one can ask if there exists a mildly mixing system such that the hyperspace is positively 1-rigid? We do not know whether this is true.

\item\label{rem:2} Note that if $(X,T)$ is positively 2-rigid then it has no asymptotic pairs, hence $T$ is a homeomorphism. Hence by Proposition~ \ref{fact}(2) if $(K(X),T_K)$ is pointwise recurrent then $(X,T)$ is an invertible t.d.s..
\end{enumerate}
\end{rem}
Combining Theorem~\ref{uni. rigid} and Proposition \ref{fact} we obtain the following.
\begin{cor} \label{two-sided UR}
Let $(X,T)$ be an invertible t.d.s. The following are equivalent:
\begin{enumerate}
\item $(X,T)$ is uniformly rigid;
\item $(K(X),T_K)$ is uniformly rigid;
\item $(K(X),T_K)$ is rigid;
\item $(K(X),T_K)$ is weakly rigid;
\item $(K(K(X)),T_K)$ is pointwise recurrent.
\end{enumerate}
\end{cor}

By \cite[Propostion 6.7]{GM89} we know that uniform rigidity is equivalent to pointwise recurrence
on $(K(X),T_K)$ when $(X,T)$ is a zero-dimensional and minimal t.d.s. Later we would show this is
also true for some class of minimal distal systems or the case of $X$ being countable. By what we have
proved it is easy to see
$$X\ \text{uniform rigidity}\ \Longrightarrow K(X)\ \text{pointwise recurrence} \Longrightarrow X\ \text{weak rigidity}$$

In Example \ref{ex1}, $(X,T)$ is weakly rigid and $K(X)$ is not pointwise recurrent. The unsolved
problem is the following

\begin{prob}\label{pro-1}
Does there exist a t.d.s. $(X,T)$ such that the hyperspace is pointwise recurrent, but $(X,T)$ is not
uniformly rigid?
\end{prob}

We strongly believe that such an example exists, though we could not provide one at this moment.

\medskip
Now we consider the (topological) entropy of hyperspace. A remarkable result
by Glasner and Weiss \cite{GW95} is that there exists a minimal system $(X,T)$ of zero entropy
with a minimal subsystem $(Y,T_K)$ of $(K(X),T_K)$ whose entropy is positive. From Remark \ref{rem}\eqref{rem:2}
and the fact that weakly rigid t.d.s. has zero entropy, we know that if $(K(X),T_K)$ is pointwise recurrent,
then $X$ has zero entropy. So we ask if the entropy of  $(K(X),T_K)$ is also zero in this case? We will answer this question
affirmatively by showing a stronger result. To start
with we need the notion of {\it locally recurrent system}, which can be found in Gottschalk and Hedlund's book \cite{GH55}.

\begin{de}
Let $(X,T)$ be an invertible t.d.s. A point $x\in X$ is called \textit{locally positively recurrent} if for each neighborhood $U$ of $x$, there are a neighborhood $V$ of $x$ and an $IP$-set $Q\subset \N$ such that $T^l V\subset U$ for all $l\in Q\bigcup \{0\}$. We say $(X,T)$
is a \textit{locally positively recurrent system} if all points in $(X,T)$ are locally positively recurrent.

In a similar manner we can define \textit{locally negatively recurrent} and \textit{locally recurrent}  points and systems.
\end{de}

\begin{rem}
\begin{enumerate}
\item
It is no hard to check the following examples:
\begin{enumerate}[(a)]
\item the Denjoy minimal system $(X,T)$ is locally recurrent
but not $2$-rigid. To see this it is enough to note that factor map
$\pi\colon X\to \T^1$ defninig Denjoy minimal system is semi-open;
\item $(\T^2,T)$ in Example \ref{ex1} is weakly rigid but not locally
recurrent. Simply for any $\ep>0$ diameter of the set $(-\ep,\ep)\times \{y\}$ starts to exceed $1/2$ after sufficiently many iterations of $T$;
\item the rigid but not uniformly rigid system described in \cite{GM89} is also not
locally recurrent (it is sufficient to check the point $(1,0)$ and its neighborhoods).
\end{enumerate}

\item From the definition we know that the local recurrence of $X$ implies its $1$-rigidity. Also,
we observe that if $K(X)$ is 1-rigid, then $X$ is locally recurrent. Since the proof is similar to
the one in Theorem \ref{p.w. min.}, we omit the simple verification.

%\item If $(X,T)$ is positively locally recurrent then it is locally recurrent and invertible {\color{red} Why invertible???}
	\end{enumerate}
\end{rem}

Topological entropy of a t.d.s. $(X,T)$, denoted by $h_{\textrm{top}}(X,T)$, measures the complex of the system.
The notion of an entropy pair was introduced by Blanchard in \cite{B93}.
Among other things, Blanchard showed that a t.d.s. has positive entropy
if and only if there exists an entropy pair. In \cite{HY06} a characterization of an entropy pair was obtained
using interpolating set. This approach was further extended in \cite{KL07} after reformulation using the notion of an independence set.
The following fact is a special case of \cite[Lemma~3.4]{KL07}.

\begin{prop} \label{ent. pair}
Let $(X,T)$ be a t.d.s. Then $(x_1,x_2)\in X^2\setminus \Delta_2$ is an entropy pair if and only if for
each neighborhood $U_i$ of $x_i$ for $i=1,2$, there is an independence set of positive density for
$(U_1,U_2)$, i.e. there is a subset $J\subset \Z$ with positive density such that for any
non-empty finite subset $I\subset J$, we have
$$\bigcap_{i\in I}\nolimits T^{-i}U_{s(i)}\neq\emptyset$$
for any $s\in \{1,2\}^I$.
\end{prop}

After this brief introduction into theory of entropy pairs we are ready to prove the following.

\begin{thm}\label{LR}
If $(X,T)$ an invertible locally recurrent t.d.s., then $h_{\emph{top}}(K(X),T_K)=0$ and hence $h_{\emph{top}}(X,T)=0$.
\end{thm}

\begin{proof}
Assume that $(X,T)$ is locally recurrent. We aim to prove $h_{\textrm{top}}(K(X),T_K)=0$. % by reduction to absurdity.
Assume on the contrary that $h_{\textrm{top}}(K(X),T_K)>0$. Then there exists a non-diagonal entropy pair
$(K_1,K_2)\in K(X)^2$ (e.g. see \cite{B93}).  Since $K_1\neq K_2$ without loss of generality we may assume that there is $z\in K_2\setminus K_1$. Then there exists an open set such that $V_1\ni z$ and $K_1\cap \ol{V_1}=\emptyset$. Then there are also open sets $U_1,\ldots,U_m$, $V_2,\ldots, V_n$ such that $V_1\cap U_1=\emptyset$ for every $i=1,\ldots,m$ and $K_1\in \U=\langle U_1,U_2,\dots, U_m \rangle$
and $K_2\in \V=\langle V_1,V_2,\dots, V_n \rangle$.
By assumption $(X,T)$ is locally recurrent, so $z$ is positively locally recurrent or negatively locally recurrent.
But by \cite[Lemma~3.2]{KL07} we see that $(K_1,K_2)$ is an entropy pair also for $T^{-1}$, hence without loss of generality we may
assume that $z$ is positively locally recurrent.
Then there are
an open neighborhood $W_0\subset V_1$ of $z$ and an $IP$-set $Q\subset\N$ with
$$T^l W_0\subset V_{1}\ \text{for all}\ l\in Q\cup \{0\}.$$
Then we can also find non-empty open sets $W_1,\dots,W_s$
such that $$\W=\langle W_0,W_1,\dots, W_s,V_{2},\ldots,V_n\rangle$$ forms
another neighborhood of $K_2$. Let $J$ be a positive density subset of $\N$ associated to $(\U,\W)$ and $T_K$ by Proposition \ref{ent. pair}.

By Lemma \ref{pubd-ip} there is $l\in Q$ such that
$\BD^*(J\cap (J-l))>0$, so there exists at least one $q\in J\cap (J-l)$, Then $q\in J$, $q+l\in J$ and since $J$ is an independence set for $(\U,\W)$, we have
$T^{-q}\W\cap T^{-{(q+l)}}\U\neq \emptyset.$ If $K\in T^{-q}\W\cap T^{-{(q+l)}}\U$
then $T^q K\in \W, T^{q+l} K\in \U$ which in other words mean
$$T^qK\cap W_0\neq \emptyset\ \  \mbox{and}\ \  T^{q+l} K\subset \bigcup_{i=1}^m U_i.$$
Observe that this leads to a contradiction, because
$$ \emptyset = \bigcup_{i=1}^m U_i\cap V_1 \supset T^{q+l} K\cap V_{1}\supset T^{q+l} K\cap T^l W_0 = T^l(T^q  K\cap W_0)\neq \emptyset.$$
We have just proved that $(K(X),T_K)$ does not have entropy pair, hence its entropy is zero.
\end{proof}

As a direct consequence we have:

\begin{cor}\label{zero ent.}
If $(K(X),T_K)$ is pointwise recurrent, then $h_{\emph{top}}(K(X),T_K)=0$.
\end{cor}

\begin{rem}
It is a long open question \cite{W95} whether 2-rigidity implies zero entropy.
Note that under the additional assumption that $X$ is positively 2-rigid, the answer is
affirmative (\cite{W95, BHR, Hoch}).
Theorem \ref{LR} indicates that this question has a
positive answer whenever $X\times X$ is locally recurrent (stronger than $X\times X$ is pointwise recurrent, i.e. $X$ is 2-rigid).
\end{rem}

Related to entropy we have the following questions.

\begin{prob} Let $(X,T)$ be a t.d.s. Denote by $K^{n+1}(X)=K^{n}(K(X)),\, n\ge 0$.
\begin{enumerate}
\item Is it true that $h_{\emph{top}}(K^n(X),T_K)=0, \, \forall n\ge 2$ if $(K(X),T_K)$ is pointwise recurrent?

\item Is is true that $h_{\emph{top}}(K(X),T_K)=0$ if $(X,T)$ is weakly rigid?
%    (see Proposition~ \ref{fact} and Corollary~ \ref{zero ent.} )
\end{enumerate}
\end{prob}

\subsection{Minimal distal systems}

In this subsection we devote to the problem of the equivalence between uniform rigidity and pointwise
recurrence on hyperspace. A partial positive answer is given below for a special class of minimal distal systems.

\begin{thm}\label{1-rigid}
Assume that $(X,T)$ is a skew product of a compact metric group $G$ over a minimal
equicontinuous system $(Y,S)$. Then the following statements are equivalent:
\begin{enumerate}
\item\label{1-rigid:1} $Y\times \{e\}$ is recurrent for $T_K$, where $e$ is the unit of $G$;
\item\label{1-rigid:2} $Y\times \{g\}$ is recurrent under action of $T_K$ for some $g\in G$;
\item\label{1-rigid:3} $(K(X),T_K)$ is pointwise recurrent;
\item\label{1-rigid:4} $(X,T)$ is uniformly rigid.
\end{enumerate}
\end{thm}

\begin{proof}
By Theorem~\ref{uni. rigid} we obtain $\eqref{1-rigid:4}\Longrightarrow\eqref{1-rigid:3}\Longrightarrow\eqref{1-rigid:1}\Longrightarrow\eqref{1-rigid:2}$, so it remains to show $\eqref{1-rigid:2}\Longrightarrow\eqref{1-rigid:4}$.

By definition  $X=Y\times G$, $T(y,g)=(Sy,\phi(y)g)$ for any $(y,g)\in Y\times G$, where  $\phi\colon Y\to G$ is continuous, $(Y,S)$ is a minimal equicontinuous t.d.s. and $G$ is a compact metric group. Note that for any $(y,g)\in Y\times G$ and $n\in\N$ we have
$$T^n(y, g)=(S^n y, \phi(S^{n-1} y)\dots \phi(y) g).$$

Fix any $\ep >0$. Since $G\times G\to G, (g_1,g_2)\mapsto g_1g_2$ is continuous, there is $\delta>0$ such that if $d((g_1,g_2),(g_1',g_2'))<\delta$ then $d(g_1g_2,g_1'g_2')<\ep$.

Since $(Y,S^{-1})$ is also minimal equicontinuous t.d.s. and inverse of $T$ is also a skew product defined by $T^{-1}(y,g)=(S^{-1}y,(\phi(S^{-1}y))^{-1}g)$,  without loss of generality we may assume that $Y\times \{g\}$ is positively recurrent.
Let $V$ be a neighborhood of $g$ in $G$ with $\diam(V g^{-1})<\delta$. It is clear $Vg^{-1}$ is a neighborhood of $e$. Observe that
$$F_1=\{n\in\N\colon T^n_K(Y\times \{g\})\in \langle Y\times V\rangle\}$$
is an $IP$-set which implies that for $n\in F_1$ we have $T^n(Y\times \{g\})\subset Y\times V$.
Since $(Y,S)$ is minimal and equicontinuous, we can find an equivalent metric $\rho$ on $Y$ such
that $\rho(Sx,Sy)=\rho(x,y)$ for every $x,y\in Y$ (by the well known Halmos-von Neumann Theorem).
Take any $y_0\in Y$ and note that
$$F_2=\{n\in\N\colon \rho(S^ny_0,y_0)<\ep/2\}$$
is an $IP^*$ set (see \cite{Fur81}).
Take any $n\in F_1\cap F_2$ and observe that
\begin{enumerate}[(a)]
\item\label{skew1:c:a} $\rho(S^n y,y)<\ep$ for each $y\in Y$, because we have $\rho(S^{n+k}y_0,S^k y_0)<\ep/2$ for each $k\in \N$ and $(Y,S)$ is minimal.
\item\label{skew1:c:b} $ \{\phi(S^{n-1} y)\dots \phi(y)g\colon y\in Y\}\subset V$.
\end{enumerate}
By \eqref{skew1:c:b} we get that
\begin{equation*}
\begin{split}
\diam \{\phi(S^{n-1}y)\dots \phi(y)\colon y\in Y\}&\le \diam (\{\phi(S^{n-1}y)\dots \phi(y)g\colon y\in Y\}g^{-1}) \\
&\le \diam (Vg^{-1})<\delta .
\end{split}
\end{equation*}

For any $(y,h)\in X$ we have $T^n(y,h)=(S^n y, \phi(S^{n-1} y)\dots \phi(y) h).$ So
$$d(T^n(y, h),(y, h))<\ep + d(\phi(S^{n-1} y)\dots \phi(y) h, h)<2\ep,$$
and this implies that $(X,T)$ is uniformly rigid.
\end{proof}

\begin{rem}
We remark here that there exists a minimal distal t.d.s. which is not equicontinuous and meets all the requirements of
Theorem \ref{1-rigid}. We forward a (slightly technical) verification of this statement to  Appendix~\ref{ex2}.
\end{rem}

\smallskip
Let $n\in \N$. We say that $(X,T)$ is a \textit{minimal distal system of class $n$} if $X$ can be realized as $n+1$ consecutive skew products, i.e. there are a compact metric Abelian group $G_0$ with minimal rotation $S=S_0$, compact metric groups $G_1,G_2, \dots, G_n$ and skew products
%$\pi_1,\pi_2,\dots,\pi_n$ such that
%$$G_0 \xla{\ \pi_{1\ }} G_0\times G_1 \xla{\ \pi_{2}\ } \cdots\xla{\ \pi_{n}\ } G_0\times G_1\times\dots\times G_n=X.$$
$$
S_i(g_0,\dots,g_i))=(S_{i-1}(g_0,\dots,g_{i-1}), \phi_i(g_0,\dots,g_{i-1}),g_i)
$$
acting on $G_0\times G_1\times \dots \times G_i$ for $i=1,\ldots,n$
such that $T=S_n$ and $X=G_0\times\dots \times G_n$.
The following result is an extension of Theorem~\ref{1-rigid}.

\begin{thm}\label{n-rigid}
Let $n\ge 2$ and assume that $(X,T)$ is a minimal distal system of class $n$. The following statements are equivalent:
\begin{enumerate}
\item\label{n-rigid:1} $(G_0\times\dots \times G_{n-1}\times \{e_n\}, \dots, G_0\times \{e_1\}\times \dots \times\{e_n\} )$ is recurrent under $T_K^{(n)}$;
\item\label{n-rigid:2} $(G_0\times\dots \times G_{n-1}\times \{g_n^1\}, \dots, G_0\times \{g_1^1\}\times \dots \times\{g_n^n\} )$ is recurrent under $T_K^{(n)}$ for some $g_j^i\in G_j,\ i=1,\dots,j$;
\item\label{n-rigid:3} $(K(X),T_K)$ is $n$-rigid;
\item\label{n-rigid:4} $(X,T)$ is uniformly rigid.
\end{enumerate}
\end{thm}

\begin{proof}
By Theorem~\ref{uni. rigid} it is clear that $\eqref{n-rigid:4}\Longrightarrow\eqref{n-rigid:3}\Longrightarrow\eqref{n-rigid:1}\Longrightarrow\eqref{n-rigid:2}$, so it remains to show $\eqref{n-rigid:2}\Longrightarrow\eqref{n-rigid:4}$.
For simplicity we present the proof only for the case $n=2$ since the proof for general case follows the same lines, however is much more technical in detail.

Let $X=G_0\times G_1\times G_2$. By assumption there are continuous maps $\phi_1\colon G_0\to G_1$, $\phi_2\colon G_0\times G_1\to G_2$.
Let $S_0\colon G_0\to G_0$, $S_1\colon G_0\times G_1\to G_0\times G_1$. Then for $(g_0,g_1,g_2)\in X$ we have
$$T(g_0,g_1,g_2)=(S_1(g_0,g_1), \phi_2(g_0,g_1)g_2).$$ Thus
$$T^n(g_0,g_1,g_2)=(S_1^n(g_0,g_1), \phi_2(S_1^{n-1}(g_0,g_1))\dots \phi_2(S_1(g_0,g_1)) \phi_2(g_0,g_1)g_2).$$
Note that
$$S_1^k(g_0,g_1)=(S_0^k\ g_0, \phi_1(S_0^{k-1}g_0)\dots \phi_1(g_0)g_1).$$

Let $\ep >0$. We choose $\delta>0$ such that if $(h_1,h_2), (h_1',h_2')\in G_1\times G_1 (\textrm{resp. } G_2\times G_2 )$ with $d((h_1,h_2),(h_1',h_2'))<\delta$ then $d(h_1h_2,h_1'h_2')<\ep$.
Let
$p=(G_0\times \{g_1^1\}\times \{g_2^1\}, G_0\times G_1\times \{g_2^2\})$, $V_1,V_2,V_2'$
be neighborhoods of  $g_1^1,g_2^1,g_2^2$ respectively with $\diam (V_1(g_1^1)^{-1})<\delta$ and $\diam (V_2'(g_2^2)^{-1})<\delta$. Put $V_p= (\langle G_0\times V_1\times V_2\rangle, \langle G_0\times G_1\times V_2'\rangle)$. Then without loss of generality,
$$F_1=\{n\in\N\colon (T\times T)^n(p)\in V_p\}$$
is an $IP$-set. Since $(G_0,S_0)$ is a minimal group rotation then by Halmos-von Neumann Theorem it is equicontinuous.
%({\color{red} I do not know where the Halmos-von Neumann Theorem first appeared})
So for some fixed $g_0'\in G_0$ we have
$$F_2=\{n\in\N\colon \rho(S^n g_0',g_0')<\ep/2\}$$ is an $IP^*$-set, where as in the proof of Theorem~\ref{1-rigid}, $\rho$ is an equivalent metric on $Y$ which does not increase distance under iteration of $S$.

For each $n\in F_1\cap F_2$ we have
\begin{enumerate}[(a)]
\item\label{skew2:c:a} $\rho(S_0^n\ g_0,g_0)<\ep$ for each $g_0\in G_0$, because we have $\rho(S_0^{n+k}g_0',S_0^k g_0')<\ep/2$ for each $k\in \N$ and $(G_0,S_0)$ is minimal;
\item\label{skew2:c:b} $ \{\phi_1(S_0^{n-1} g_0)\dots \phi_1(g_0)g_1^1\colon g_0\in G_0\}\subset V_1$;
\item\label{skew2:c:c} $ \{ \phi_2(S_1^{n-1}(g_0,g_1))\dots \phi_2(g_0,g_1)g_2^2\colon (g_0, g_1)\in G_0\times G_1\}\subset V_2'$.
\end{enumerate}
From \eqref{skew2:c:b} and \eqref{skew2:c:c} we know
$$
\diam \{\phi_1(S_0^{n-1} g_0)\dots \phi_1(g_0)\colon g_0\in G_0\} \le \diam (V_1(g_1^1)^{-1})<\delta,
$$
$$
\diam \{ \phi_2(S_1^{n-1}(g_0,g_1))\dots \phi_2(g_0,g_1)\colon (g_0, g_1)\in G_0\times G_1\} \le \diam (V_2'(g_2^2)^{-1})<\delta .
$$
It implies that
$$d(T^n(g_0, g_1, g_2),(g_0, g_1, g_2))<3\ep,$$
and then $(X,T)$ is uniformly rigid.
\end{proof}

\begin{prob}\label{pro-3}
The following questions arise naturally:
\begin{enumerate}
  \item Does the equivalence of \eqref{n-rigid:3} and \eqref{n-rigid:4} in Theorem \ref{n-rigid} still hold, when $(X,T)$ is a general minimal distal system (not necessarily group skew product)?

  \item Is there a minimal distal system $(X, T)$ of class $n$ such that $(K(X),T_K)$ is $(n-1)$-rigid but $(G_0\times\dots \times G_{n-1}\times \{e_n\},\dots, G_0\times \{e_1\}\times \dots \times\{e_n\})$ is not recurrent?

\end{enumerate}
\end{prob}

We remark that if Problem \ref{pro-3} (2) has a positive answer, then $(K(X),T_K)$ is $(n-1)$-rigid, and not $n$-rigid.
Particularly, if Problem \ref{pro-3} (2) has a positive answer for $n=2$, then Problem \ref{pro-1}
will be solved, i.e. there is a t.d.s. $(X,T)$ such that $(K(X),T_K)$ is pointwise recurrent and at the same time $(X,T)$ is not uniformly
rigid.

%\begin{rem} We have the following remarks:

%\begin{enumerate}

%\item
%

%\item We may define a \textit{minimal distal system of class $\infty$} as the inverse limit
%of minimal distal systems of class $n$. By the similar argument we may show that if $(X,T)$ is a minimal distal system of class $\infty$
%then $(K(X),T_K)$ is weakly rigid if and only if $(X,T)$ is uniformly rigid.
%{\color{red}Define $\infty$ class and weak rigid ---uniformly rigid}
%\end{enumerate}
%\end{rem}

\subsection{The countable case}

In this section, we discuss pointwise recurrence induced on hyperspace when $X$ is countable.

First, we recall the notion of the derived set of $X$ of order $\alpha$. A point $x$ of $X$ is an \textit{accumulation point} of the set $X$ if $x\in \ol{X\setminus \{x\}}$. The set of accumulation points of $X$ is said to be the \textit{derived set} of $X$, denote as $X^*$. The derived set of $X$ of order $\alpha$ is defined by the conditions: $X^{(1)}=X^*, \ X^{(\alpha+1)}=(X^{(\alpha)})^*$ and $X^{(\lambda)}=\bigcap_{\alpha<\lambda} X^{(\alpha)}$ if $\lambda$ is a limit ordinal number. We put $d(X)=\alpha$ if $X^{(\alpha)}\neq\emptyset$ and $X^{(\alpha+1)}=\emptyset$.  Here $d(X)$ is called the \textit{derived degree} of $X$.

It is well known that a compact metric space $X$ is a countable set if and only if $d(X)$ exists and it is a countable ordinal number. In this case, if $d(X)=\alpha$, then $X^{(\alpha)}$ is a finite set. More details can be found in the book by Kuratowski \cite[p. 261]{Ku68}.

\begin{thm}\label{countable}
Let $(X,T)$ be an invertible t.d.s. acting on a countable space $X$. The following statements are equivalent:
\begin{enumerate}
\item\label{countable:1} $(K(X),T_K)$ is pointwise recurrent;
\item\label{countable:2} $(K(X),T_K)$ is uniformly rigid;
\item\label{countable:3} $(K(X),T_K)$ is rigid;
\item\label{countable:4} $(K(X),T_K)$ is weakly rigid;
\item\label{countable:5} $(X,T)$ is uniformly rigid;
\item\label{countable:6} $(X,T)$ is pointwise periodic.
\end{enumerate}
\end{thm}
\begin{proof}
We have already proved the equivalence of \eqref{countable:2}--\eqref{countable:5}
in Theorem~\ref{uni. rigid}, and clearly \eqref{countable:5} implies \eqref{countable:1}.
To see that \eqref{countable:1} implies \eqref{countable:6} take any  $x\in X$ and put
$A=\ol{\textrm{orb}(x,T)}$ or $A=\ol{\textrm{orb}(x,T^{-1})}$ if $x$ is positively or
negatively recurrent, respectively. Since $A$ is compact and countable, it is not a
perfect set and hence has at least one isolated point. This shows that $x$ must be a periodic point.

It remains to show $\eqref{countable:6}\Longrightarrow\eqref{countable:5}$. If $X$ is finite,
\eqref{countable:5} is obvious. When $X$ is infinite, there is a countable ordinal number $\alpha$ such that $d(X)=\alpha$.

We denote $A=X^{(\alpha)}=\{a_1,a_2,\ldots,a_m\}$ and may further assume that each $a_i \in A$ is a fixed point,
since in the proof of uniform rigidity we can always replace $T$ by its higher iterate. We can also choose
a fixed (arbitrarily small) $\ep>0$ such that $B_\ep(A)=\ol{B_\ep(A)}$, $B_\ep(a_j)=\ol{B_\ep(a_j)}$
for each $a_j \in A$ and $TB_{\ep}(a_i)\cap B_{\ep}(a_j)=\emptyset$ when $i \neq j$. Now we divide $X$ into three subsets:
\begin{enumerate}[(I)]
	\item $X_1=\{x\in X\colon x\notin B_\ep(A)\}$,
	\item $X_2 =\bigcup_{j=1}^{m}X_{2,j}, \ \mbox{where} \ \ X_{2,j}=\{x\in X: x \in B_{\ep}(a_j) \ \mbox{and}\ \textrm{orb}(x,T) \not\subset B_{\ep}(a_j)\}$,
	\item $X_3 =\bigcup_{j=1}^{m}X_{3,j}, \ \mbox{where} \ \ X_{3,j}=\{x\in X: x \in B_{\ep}(a_j) \ \mbox{and}\ \textrm{orb}(x,T) \subset B_{\ep}(a_j)\}.$
\end{enumerate}

Note that $X_1$ is closed and $X_1 \cup X_2 $ is invariant under $T$.
We claim that

\medskip
\noindent {\bf Claim A: } There exists $n\in\N$ such that $\widetilde{X}_2 =\emptyset$, where $\widetilde{T}=T^n$ and
\begin{enumerate}[(I\')]
	\item $\widetilde{X}_1=\{x\in X\colon x\notin B_\ep(A)\}$,
	\item $\widetilde{X}_2 =\bigcup_{j=1}^{m}\widetilde{X}_{2,j}, \ \mbox{where} \ \ \widetilde{X}_{2,j}=\{x\in X: x \in B_{\ep}(a_j) \ \mbox{and}\ \textrm{orb}(x,\widetilde{T}) \not\subset B_{\ep}(a_j)\}$,
	\item $\widetilde{X}_3 =\bigcup_{j=1}^{m}\widetilde{X}_{3,j}, \ \mbox{where} \ \ \widetilde{X}_{3,j}=\{x\in X: x \in B_{\ep}(a_j) \ \mbox{and}\ \textrm{orb}(x,\widetilde{T}) \subset B_{\ep}(a_j)\}.$
\end{enumerate}
( Please see Appendix B for the proof of the Claim A.)

\medskip
By {\bf Claim A} it is clear that $\widetilde{X}_1=X\setminus B_{\ep}(A)$ is closed and invariant under $\widetilde{T}$ and $\widetilde{X}_{3,t}=B_{\ep}(a_t)$ for $t=1,2,\dots,m$.
Therefore $(\widetilde{X}_1, \widetilde{T}|_{\widetilde{X}_1})$ is a subsystem and $d(\widetilde{X}_1)<\alpha$.

Now we shall prove $\eqref{countable:5}$ by induction on $\alpha$. First assume that $\alpha=1$. By {\bf Claim A} we can find $n\in \N$ such that $\widetilde{X}_2 =\emptyset$,
$(\widetilde{X}_1, \widetilde{T}|_{\widetilde{X}_1})$ is a subsystem and $\widetilde{X}_1$ is finite.
Let $r$ be the common period of each point in $\widetilde{X}_1$ under $\widetilde{T}$. Then we have $d(\widetilde{T}^rx,x)=0<2\ep$ for $x \in \widetilde{X}_1$.
And for any $t\in\{1,2,\dots,m\}$, it is clear that $d(x,\widetilde{T}^r x)<2\ep$  for each $x \in \widetilde{X}_{3,t}=B_{\ep}(a_t)$.
Hence $d(x,\widetilde{T}^r x)<2\ep$ for each $x\in X$. That is , $\widetilde{T}$ is uniformly rigid and hence so is $T$, proving $\eqref{countable:5}$ for the case of $\alpha=1$.

Next we assume that $\eqref{countable:5}$ holds for all cases with $d(X)<\alpha$, and the goal is to prove
$\eqref{countable:5}$ is still true for $d(X)=\alpha$.
By {\bf Claim A} we can find $n\in \N$ such that $\widetilde{X}_2 =\emptyset$,
$(\widetilde{X}_1, \widetilde{T}|_{\widetilde{X}_1})$ is a subsystem and $d(\widetilde{X}_1)<\alpha$.
Thus $(\widetilde{X}_1, \widetilde{T}|_{\widetilde{X}_1})$ is uniformly rigid by the inductive assumption.
Moreover, note that for any $t\in\{1,2,\dots,m\}$ we have $d(x,\widetilde{T}^r x)<2\ep$
for each $x \in \widetilde{X}_{3,t}=B_{\ep}(a_t)$ and $r\in\N$. So $\widetilde{T}$ is uniformly rigid,
and then so is $T$. That is $\eqref{countable:5}$ holds for $d(X)=\alpha$. This ends the proof.
\end{proof}

\section{Dense recurrent points in the hyperspace}\label{weak}
In this section we discuss the situation when $(K(X),T_K)$ has a dense set of recurrent points.

%\subsection{Dense recurrent closed subsets}

\begin{thm}\label{dense}
Let $(X,T)$ be an invertible t.d.s. We have
\begin{enumerate}
\item If $(K(X),T_K)$ has a dense set of recurrent points then so does $(X,T)$;

\item If $(X,T)$ is transitive then $(K(X),T_K)$ has a dense set of recurrent points.

\end{enumerate}
\end{thm}

\begin{proof}
Assume that $(K(X),T_K)$ has dense recurrent points. Let $U,V\subset X$ be two non-empty open subsets of $X$ with $\ol{V}\subset U$. Then $\langle V \rangle$ is non-empty open in $K(X)$. By assumption, without loss of generality, there exists a positively recurrent point $A\in\langle V \rangle$ and an $IP$-subset $Q$ of $\N$ such that $T_K^l A\in \langle V\rangle$, i.e, $T^l A\subset V$, for all $l\in Q$. By {\cite[Lemma 2.3]{DSY12}}, $\ol{V}\cap \textrm{Rec}(T)\neq\emptyset$, so $U\cap \textrm{Rec}(T) \neq\emptyset$ and (1) holds.

Now assume that $(X,T)$ is transitive and fix any non-empty open sets $V_1,\ldots,V_n\subset X$. Let $x\in X$ be a point with dense orbit and let $k_1,\ldots,k_n\in \N$
be such that $T^{k_i}x\in V_i$. If we put $A=\{T^{k_i}x : i=1,\ldots,n\}$ then $A\in \langle V_1,\ldots, V_n\rangle$ and since $x$ is positively recurrent, then
$$
\liminf_{j\to \infty} d_H(A,T_K^j A)=0
$$
completing the proof.
\end{proof}

From the above proof, we obtain that the self-product of a transitive t.d.s. $(X,T)$ has dense recurrent points. But we note that there are tansitive t.d.s. $(X,T)$ and $(Y,S)$ such that $X\times Y$ does not have a dense set of recurrent points (for example, see \cite[theorem 6.4]{DSY12}). These dynamical systems can be used to produce the following example.

\begin{exam}
If we denote disjoint union $Z=X\cup Y$ and define t.d.s. $(Z,F)$ by $F|_X=T$, $F|_Y=S$ then clearly $(Z,F)$ has dense positively recurrent points, but there
are open sets $U\subset X$, $V\subset Y$ such that $(F^{(2)})^j (U\times V)\cap U\times V=\emptyset$. Then $\langle U,V\rangle$ does not contain recurrent points of $F_K$,
in particular $(K(Z),F_K)$ does not have dense recurrent points.
\end{exam}

\subsection{Dense distal closed subsets}
It is easy to check that if $(X,T)$ has dense distal points, then so is $(K(X),T_K)$. In this subsection
we shall show the converse is not true, which answers a question left open in \cite{LYY13}.

Before we proceed with the example, let us first recall some basic notations on symbolic dynamics. Let $\Sigma_2=\{0,1\}^\N$ equipped
with the product topology. Then $\Sigma_2$ is compact, and the \textit{shift map} $\sigma\colon \Sigma_2\to \Sigma_2$ defined by $\sigma(x)_n=x_{n+1}$ for $n\in\N$ is continuous.
Any non-empty, closed and $\sigma$-invariant subset $X\subset \Sigma_2$ is called a \emph{subshift} and is identified with the subsystem $(X,\sigma)$.

Fix $n\in\N$, we call $w\in\{0,1\}^n$ a \textit{word of length $n$} and write $|w|=n$ and we denote
$|w|_a=\#\{i\in\N\colon w_i=a\}$ the number of occurrences of symbol $a$ in the word $w$. For any
two words $u=u_1u_2\dots u_n$ and $v=v_1v_2\dots v_m$, the \textit{concatenation} of $u,v$ is
defined by $uv=u_1u_2\dots u_nv_1v_2\dots v_m$. Analogously  $u^m$
is defined by the concatenation of $m$ copies of $u$ for some $m\in\N$, and $u^\infty$ the
infinite concatenation of $u$. Let $X$ be a subshift of $\Sigma_2$ and $x=x_1x_2\dots\in X$.
We say that a word $w=w_1w_2\dots w_n$ appears in $x$ at position $t$ if $x_{t+j-1}=w_j$ for $j=1,2,\dots, n$.
By $L(X)$  we denote the \textit{language} of subshift $X$, that is the set consisting of all words
that can appear in some $x\in X$, and we write $L_n(X)$ as the set of all words of length $n$
in $L(X)$. For any word $u\in L_n(X)$ its \emph{cylinder set} is defined by $[u]=\{x\in X \colon x_1x_2\dots x_n=u\}$.
Note that all cylinder sets $\{[u]\colon u\in L(X)\}$ form a basis of the topology of $X$.

%\subsubsection{Relatively weakly mixing minimal systems}
The basis for our construction will be a sequence of weakly mixing minimal subshifts provided by the following lemma.
\begin{lem}\label{lem:specWMsubshift}
Fix an $\ep>0$ and a non-empty word $w$ with at least one occurrence of $1$.
Then there exists a subshift $X=X(w,\ep)\subset \Sigma_2$ such that (for some integer $n>0$) the following conditions are satisfied:
\begin{enumerate}[(i)]

\item\label{mwm:c1} $X$ is minimal and $\# X=\infty$;

\item\label{mwm:c2} there are weakly mixing minimal systems $(D_0, \sigma^n),\dots, (D_{n-1}, \sigma^n)$ (not necessarily distinct) such that $X=\bigcup_{i=0}^{n-1}D_i$ and $\sigma(D_i)=D_{i+1 (\emph{mod } n)}$;

\item\label{mwm:c3} $D_0\subset [w]$;

\item\label{mwm:c4} for some integer $k>0$ we have $k,k+1\in N_\sigma([w],[w])$ in $X$, i.e. there are words $u,v$
	such that $|u|=|v|+1$ and $wuw,wvw\in L(X)$;

\item\label{mwm:c5} there is $N\leq n$ such that $\frac{|v|_1}{m}<\ep$ for every $v\in L_m(X)$ and $m\geq N$;

\item\label{mwm:c6} for every $v\in L(X)$ with $|v|\geq |w|$ we have $\frac{|v|_1}{|v|}\leq \frac{|w|_1}{|w|}$.

 \end{enumerate}
\end{lem}

\begin{proof}
We will present a construction of subshift $X=X(w,\ep)$. Let $s=|w|$ and let $t>3s$ be such that
$2s/(s+t)<\ep$. We also assume that $2s/t<|w|_1/|w|$. Put $u_0=w 0^t w 0^{t+1}w 0^t$ and
$v_0=w 0^{3t+2s+1}$ and note that $|u_0|=|v_0|$. Denote $n=|u_0|$. Set $u_1=u_0 v_0 u_0$
and $v_1=u_0 v_0$. Then we recursively define
$$u_{k+1}=u_k v_k u_k u_k\ \text{and}\ v_{k+1}=u_k v_k v_k u_k$$
for all $k\geq 1$. Observe that $|u_1|-|v_1|=n$
and therefore  $|u_{k+1}|-|v_{k+1}|=|u_{k}|-|v_{k}|=n$ for every $k\geq 1$.
Moreover, $n||u_k|$ and $n||v_k|$ for each $k\ge 1$.

Let $z=\lim_{k\to +\infty} u_k$ and $X=\ol{\textrm{orb}(z, \sigma)}$ , i.e.
each $u_k$ is a prefix of $z$. Note that each $u_{k+1}$ and $v_{k+1}$ is a
concatenation of $u_k$ and $v_k$ and both words appear at least once. Therefore,
since $z$ can be presented as an infinite concatenation of $u_{k+1}$ and $v_{k+1}$,
we immediately obtain that $N_\sigma(z,[u_k])$ is positively syndetic for every $k$.
This shows that $z$ is a minimal point, and hence $X$ is minimal.

Put $D_0=\ol{\textrm{orb}(z, \sigma^n)}$ and $D_i=\sigma^i(D_0)$ for $i=1,\dots,n-1$.
It is clear that each $(D_i,\sigma^n)$ is minimal and $\sigma$ permutes sets $D_0,\dots, D_{n-1}$
periodically. Note that $z$ is an infinite concatenation of $u_0$ and $v_0$. Furthermore,
both $u_0,v_0$ have length $n$ and both have word $w$ as a prefix. Therefore $\sigma^{ni}(z)\in [w]$
for every $i\geq 0$, in particular $D_0\subset [w]$ and so \eqref{mwm:c3} is satisfied.
Furthermore $u_k$ is both prefix and suffix of $u_{k+1}=u_k v_k u_k u_k$ and $v_{k+1}=u_k v_k v_k u_k$. This shows that
for some $m$ we have $m,m+1\in N_{\sigma^n}([u_k],[u_k])$ since $|u_{k}|-|v_{k}|=n$. As each $D_i$ is minimal,
by Lemma \ref{w.m.}(1) we have $(D_i,\sigma^n)$ is weakly mixing for all $i$.
This proves \eqref{mwm:c2}. By the definition of $u_0$ we also have $t+s,t+s+1\in N_\sigma([w],[w])$ which gives \eqref{mwm:c4}.
Additionally observe that $D_0$ has at least two points (starting with $u_0$ and $v_0$)
hence by weak mixing $(D_0,\sigma^n)$ is infinite, in particular $X$ is infinite showing \eqref{mwm:c1}.

Finally, put $N=s+t$ and fix any $v\in L_m(X)$ for some $m\geq N$. We can write $m=jN+r$ with $r<N$.
Clearly, $v$ must appear at some position in $z$. By the definition of $u_0,v_0$ every word of $N$
consecutive symbols in $z$ can have at most $|w|_1\leq s$ occurrences of symbol $1$. Then
$$\frac{|v|_1}{|v|}\leq \frac{(j+1)|w|_1}{j(s+t)}\leq \frac{2s}{s+t}<\ep$$
and so \eqref{mwm:c5} holds.
This also shows that if $|v|\geq N$ then $\frac{|v|_1}{|v|}<\frac{2s}{t}\leq\frac{|w|_1}{|w|}$.
But if $|v|\leq N$ then $|v|_1\leq |w|_1$ and since in \eqref{mwm:c6} we need only to
consider $|v|\geq |w|$ then also in this case $|v|_1/|v|\leq |w|_1/|w|$.
This shows \eqref{mwm:c6}, that is for every $v\in L(X)$ with $|v|\geq |w|$ we have
$|w|_1/|w|\geq |v|_1/|v|$. The proof is completed.
\end{proof}

%\subsubsection{Weakly mixing subshift}
Now we are ready to perform the main construction.

Start with $w_1=01 0^{10}$ and $\ep_1=1/9$ and let $M_1=X(w_1,\ep_1)$
be a minimal system provided by Lemma~\ref{lem:specWMsubshift}.
For each $n\geq 1$ we put $\ep_{n}=9^{-n}$. Now we will present how to construct subshifts
$M_{n}$ inductively. For induction, assume that we already constructed subshifts $M_1,\dots, M_n$
by Lemma~\ref{lem:specWMsubshift} with words $w_1,\ldots, w_n$ and $\ep_1,\dots,\ep_n$.
We also assume that $|w_i|<|w_{i+1}|$ for each $i$ and $|w_n|\geq (n-1) 9^{n-1}$.
Enumerate words $L_{n}(\cup_{i=1}^n M_i)=\{u_1,\ldots,u_{s_n}\}$. Let
$$w_{n+1}=u_1 0^{n9^n}u_2 0^{n9^n}\dots u_{s_n} 0^{n9^n}.$$
Let $M_{n+1}$ be the minimal subshift constructed by
Lemma~\ref{lem:specWMsubshift} for word $w_{n+1}$ and $\ep_{n+1}$, i.e. $M_{n+1}=X(w_{n+1},\ep_{n+1})$.

Let $\X=\ol{\cup_{n=1}^{+\infty} M_n}$. Clearly $\X$ is closed and $\sigma$-invariant, therefore it is a subshift.

\begin{thm}\label{K(X) P-system}
The subshift $\X$ has the following properties:
\begin{enumerate}
\item if $M\subset \X$ is minimal then either $M=\{0^\infty\}$ or $M=M_n$ for some $n\geq 1$,
\item $(\X,\sigma)$ is weakly mixing, and
\item $(K(\X),\sigma_K)$ has dense periodic points.
\end{enumerate}
\end{thm}

\begin{proof}
Fix any minimal set $M\subset \X$ and assume that $M\neq \{0^\infty\}$. Fix $z\in M$ and observe that the
set $\{i \colon z_i=1\}$ is positively syndetic. We additionally assume that $z_0=1$. For $k=2,3,\dots$,
let $v_k$ denote prefix of $z$ ending with symbol $1$ and such that $|v_k|_1=k$.
Since symbol $1$ appears in $z$ syndetically, there is $\delta>0$ such that $|v_k|_1/|v_k|>\delta$ for
every $k$. Let $m\in\N$ with $1/m<\delta$.
Assume on the contrary that $M\neq M_n$ for every $n$. Then there exists $K$ such that $v_k\not\in L(M_n)$
for every $n\leq m$ and $k> K$. Take any $k>K$ and let $j$ be the minimal integer such that $v_k\in L(M_j)$.
If $|v_k|>|w_j|$ then by \eqref{mwm:c6} we obtain
$$\frac{|v_k|_1}{|v_k|}\leq \frac{|w_j|_1}{|w_j|}\leq \frac{1}{9^{j}}< \frac{1}{9^m}<\frac{1}{m},$$
which is a contradiction. But if $|v_k|\leq |w_j|$ then by the method of construction of $X(w_j,\ep_j)$
(in particular, definition of $s,t$ in the proof of Lemma~\ref{lem:specWMsubshift}) and the fact
that $v_k$ starts and ends with symbol $1$, we obtain that $v_k$ is a subword of
$$w_j=u_1 0^{(j-1)9^{(j-1)}}u_2 0^{(j-1)9^{(j-1)}}\dots u_{s_{j-1}} 0^{(j-1)9^{(j-1)}},$$
where by the definition each word $u_i\in L_{j-1}(M_r)$ for some $r<j$. By the minimality of
$j$ we see that $v_k$ cannot be a subword of any $u_i$ and since it starts and ends by symbol $1$,
it must be placed in $w_j$ in such a way that it contains at least one word $0^{(j-1)9^{(j-1)}}$.
Let $p$ be a minimal number such that for some $i$ word $v_k$ is a subword of
$u_i 0^{(j-1)9^{(j-1)}} u_{i+1} 0^{(j-1)9^{(j-1)}}\dots u_{i+p}$. Then simple calculations yield that
$$\frac{|v_k|_1}{|v_k|} \leq \frac{(p+1)(j-1)}{p(j-1)9^{(j-1)}}\leq \frac{2}{9^{j-1}}\leq \frac{2}{9^{m}}< \frac{1}{m}<\delta,$$
which is again a contradiction, proving  that $M=M_n$ for some $n\geq 1$.

In order to prove that $(\X,\sigma)$ is weakly mixing, it is enough to show that for any
$u,v\in L(\X)$ there is $k$ such that $k,k+1\in N_\sigma([u],[v])$. Without loss of generality,
we may assume that $|u|=|v|$. Let $m$ be such that $u,v\in L(\cup_{i=1}^m M_i)$. We can
extend words $u,v$ if necessary, obtaining that $|u|=|v|= m$.
Then $u,v$ are subwords of $w_{m+1}$ and then, by the definition of $M_{m+1}$ and
\eqref{mwm:c4} we obtain integer $t$ such that $t,t+1\in N_\sigma([w_{m+1}],[w_{m+1}])$.
But since both $u,v$ are subwords of $w_{m+1}$ there clearly exists $k$ such that $k,k+1\in N_\sigma([u],[v])$.
This shows that $(\X,\sigma)$ is weakly mixing by Lemma \ref{w.m.}(1).

By the same argument, if we fix any word $u\in L(\X)$ then there is $m$ such that $u$ is
a subword of $w_{m+1}$ and by \eqref{mwm:c3} there are $n$ and $D\subset M_{m+1}\cap [w_{m+1}]$
such that $\sigma^n(D)=D$. But there is also $j\geq 0$ such that $\sigma^j([w_{m+1}])\subset [u]$
which shows that there is a periodic set $\sigma^j(D)\subset [u]$. Indeed $(\X,\sigma)$ has dense periodic sets.
\end{proof}

Let $(X, T)$ and $(Y, S)$ be two t.d.s. A non-empty, closed and invariant subset $J\subset X \times Y$
is a \textit{joining} of $X$ and $Y$ if $J$ projects onto $X$ and $Y$, respectively. We say $(X, T)$
and $(Y, S)$ are \textit{disjoint} if $X \times Y$ is the only joining.

The question which t.d.s. is disjoint from all minimal systems was asked in \cite{Fur67} and Furstenberg showed that
every weakly mixing $P$-system has this property. A systematic study of the question was carried out in \cite{HY05},
where the condition of a $P$-system was weakened to a system with dense small periodic sets.
In \cite{O10, DSY12}, it was showed that each weakly mixing
system with dense distal points is disjoint from any minimal
systems; and in \cite[Theorem 5.5]{LYY13} the authors showed that if $(X, T)$ is a weakly mixing t.d.s.
and $K(X)$ has dense distal points, then $(X, T)$ is disjoint from all minimal systems. Now we point out
that there is a t.d.s. $(X, T)$ which does not have dense distal points but $(K(X),T_K) $ has the property.
It give a positive answer on a question left open in \cite{LYY13}.

\begin{cor}\label{openProb}
The above t.d.s. $(\X,\sigma)$ does not have dense distal points, however $(K(\X),\sigma_K)$ is a $P$-system. In particular, both $(\X,\sigma)$ and $(K(\X),\sigma_K)$
are disjoint from any minimal system.
\end{cor}

\begin{proof}
Note that closure of the orbit of a distal point is a minimal set. By Theorem~ \ref{K(X) P-system} the minimal
set is either $\{0^\infty\}$ or $M_n$ for some $n\geq 1$. But Lemma \ref{lem:specWMsubshift} shows that each
$M_n$ can be divided into serval infinite, relatively weakly mixing and minimal subsystems.
Applying Lemma \ref{w.m.}(2) we know $(\X,\sigma)$ does not have dense distal points (in fact $0^\infty$ is the unique distal point in $\X$). By Lemma \ref{P-system}, Theorem \ref{K(X) P-system} and \cite[Theorem 5.5]{LYY13} the result follows.
%Note that closure of the orbit of a distal point is a minimal set. By Lemma \ref{w.m.}(2), Lemma \ref{P-system}, Lemma \ref{lem:specWMsubshift}, Theorem \ref{K(X) P-system} and \cite[Theorem 5.5]{LYY13} the corollary follows.
\end{proof}

\begin{rem}
The construction in Theorem~\ref{K(X) P-system} was inspired by the technique developed in \cite{Kwi}
to prove the existence of a mixing shift space with a dense set of periodic points but without ergodic measure with full support.	
\end{rem}

\appendix

\section{Some examples}

In \cite{Dong} Dong showed that for minimal nilsystems $(X,T)$ is uniformly rigid if and only
if $(X,T)$ is equicontinuous. So one may ask if this holds for minimal distal systems. In this
section we consider a special class of distal systems and show that in this class uniform rigidity and equicontinuity are different properties.
%$(X,T)$ is uniformly rigid is
%not equivalent to that $(X,T)$ is equicontinuous.

Let $X=\T^2$ and $T\colon X\to X$ be a group extension over an irrational rotation on $\T^1$, i.e.
for any $(x,y)\in [0,1)\times [0,1)$,
\begin{equation}
T(x,y)=(x+\alpha \modone, \phi(x)+y \modone),\label{skew:T}
\end{equation}
 where $\phi\colon\R\to \R$ is continuous
with $\phi(1)-\phi(0)\in\Z$. Then for any $(x,y)\in X$ and $n\in \N$ we have
$$T^n(x,y)=(x+n\alpha \modone, \sum_{i=0}^{n-1}\phi(x+i\alpha)+y \modone).$$
It is easy to see that $(X,T)$ is distal.

\medskip
We will choose suitable $\phi$ to induce desired properties. Our construction will rely on the degree of $\phi$ (see \cite{KH95} for introduction).

\begin{prop}
If $|deg(\phi)|\ge 1$, then $(K(X),T_K)$ is not positively recurrent. In particular, $(X,T)$ is not uniformly rigid.
\end{prop}

\begin{proof}
Set
$$A_n(y)=\{\sum_{i=0}^{n-1}\phi(x+i\alpha)+y \modone: x\in [0,1)\}.$$
Let $\phi_0=\phi$ and $\phi_n\colon \R\to \R$ with $\phi_n(x)=\phi(x+n\alpha)$. If $|deg(\phi)|=d\ge 1$,
then $|deg(\phi_{n-1}+\dots+\phi_0)|=nd$. Thus for $n\ge 2$, $A_n(y)=\T^1$
for any $y\in [0,1)$. This implies that $\T^1\times \{y\}$
is not positively recurrent in $K(X)$, and hence $(X,T)$ is not uniformly rigid.
\end{proof}

Now we consider the case when $deg(\phi)=0$. If there are  $f$ and $c$ such that $\phi(x)=f(x+\alpha)-f(x)+c$.
Then $$\sum_{j=0}^{n-1}\phi(x+j\alpha)=f(x+n\alpha)-f(x)+nc.$$
It is easy to see that $(X,T)$ is equicontinuous.

\begin{exam}\label{ex2}
There is $\phi$ with $deg(\phi)=0$ such that $(X,T)$ is minimal distal non-equicontinuous and uniformly rigid.
\end{exam}

\begin{proof}
Consider $(X,T)$ given by \eqref{skew:T} with suitable $\phi\colon \R\to \R$ which we define below.
Choose a sequence $\{n_j\}_{j=1}^{\infty}\subset\N$ such that
$n_1=100,\,n_{j+1}=(n_1 n_2 \dots n_j)^3$ for $j\ge1$. Denote  $\alpha=\sum_{j=1}^{\infty}
\frac{1}{n_j}$, clearly $\alpha$ is an irrational number (since it is an infinite non-repeating decimal). We claim that such $\alpha$ also satisfies the following two conditions:
\begin{enumerate}[(a)]
\item $\left| e^{2\pi in_k \alpha}-1\right|<\frac{14}{n_k^2}$ for every $k\in \N$;
\label{cond:T:a}
\item $\left|e^{2\pi i n_kn_l\alpha}-1\right|< \frac{14}{n_k n_l}$ for every $k\in \N$ and $l\in\N\setminus \{1\}$.
\label{cond:T:b}
\end{enumerate}
Indeed, for \eqref{cond:T:a} we notice that $n_k\alpha=\sum_{j=1}^{k}\frac{n_k}{n_j}
+\sum_{j=k+1}^\infty\frac{n_k}{n_j}$ and $ \sum_{j=k+1}^{\infty}\frac{n_k}{n_j}<\frac{2}{n_k^2}$ and $n_{k+1}=(n_1 n_2 \dots n_k)^3$, then
$$
\left| e^{2\pi in_k \alpha}-1\right|=| e^{2\pi i\sum_{j=k+1}^\infty\frac{n_k}{n_j}}-1|< 2\pi \sum_{j=k+1}^{\infty}\frac{n_k}{n_j}<\frac{14}{n_k^2}
$$
for every $k\in \N$. Also, we can take similar arguments to check \eqref{cond:T:b}.
Simply note that if we assume that $l \ge k$ then $n_ln_k\alpha=\sum_{j=1}^{l}\frac{n_ln_k}{n_j}+\sum_{j=l+1}^{+\infty}\frac{n_ln_k}{n_j}$
and $\sum_{j=l+1}^{+\infty}\frac{n_ln_k}{n_j}<n_ln_k\frac{2}{n_{l+1}}=\frac{2n_ln_k}{(n_1 n_2 \dots n_l)^3}<\frac{2}{n_k n_l}$, which gives
$$
\left| e^{2\pi in_k n_l \alpha}-1\right|< 2\pi \sum_{j=l+1}^{\infty}\frac{n_k n_l}{n_j}<2\pi \frac{2}{n_k n_l}<\frac{14}{n_k n_l}
$$

Let $n_{-k}=-n_k,\, n_0=0$ and for each $x\in \R$ put
$$\phi(x)=\sum _{k=-\infty}^{+\infty} (e^{2\pi i n_k \alpha}-1)  e^{2\pi i n_k x}.$$
Observe that for every $x\in \R$ we have $\phi(x)\in \R$, hence the function $\phi\colon \R\to \R$ is well defined.
By \eqref{cond:T:a} we have $\left|\phi(x)\right|\leq 28\sum_{k=1}^\infty \frac{1}{n_k^2}<\infty$, so the series
converges uniformly on $[0,1]$ and then $\phi$ is continuous. It follows directly from the definition that $deg(\phi)=0$.

Similarly to the discussion in \cite[pp.73-75]{A88} we can see that the t.d.s. $(\T^2,T)$ is a
minimal distal system. Moreover, we claim that it is also non-equicontinuous. To see this,
consider $(x_{n_l},0)\ra (0,0),\,( l\ra \infty)$ with $x_{n_l}={1}/({n_1n_l})$, and choose
$m_l={n_l^3}/{n_1}$, $\delta=1/{1000}$. It suffices to show
\begin{equation}
  d(T^{m_l}(x_{n_l},0),T^{m_l}(0,0))>\delta.\label{cond:*}
\end{equation}

To validate this, we first note that for each $x\in \T^1$ and $n\in \N$,
\begin{equation*}
\begin{split}
  \sum_{t=0}^{n-1}\phi(x+t\alpha) & =\sum _{k=-\infty}^{+\infty} (e^{2\pi i n_k \alpha}-1) \sum_{t=0}^{n-1}e^{2\pi i n_k(x+t\alpha)} \\
    &=\sum_{k=-\infty}^{+\infty}e^{2\pi i n_k x}\sum_{t=0}^{n-1}(e^{2\pi i n_k(1+t)\alpha}- e^{2\pi i n_k t\alpha})\\
    &=\sum _{k=-\infty}^{+\infty}(e^{2\pi i n_kn\alpha}-1)e^{2\pi i n_kx}.
\end{split}
\end{equation*}
If $k>l$, then $n_k x_{n_l}$ is an integer. Also, for each $1\le k\le l$, we have
$$
n_k m_l\alpha=\sum_{j=1}^{l}\frac{n_k m_l}{n_j}+\sum_{j=l+1}^{\infty}\frac{n_k m_l}{n_j}.
$$
 This implies that
\begin{eqnarray*}
& &\sum_{t=0}^{m_l-1}\phi(0+t\alpha)-\sum_{t=0}^{m_l-1}\phi(x_{n_l}+t\alpha)\\
%&=&\sum_{k=-\infty}^{+\infty}(e^{2\pi i n_km_l\alpha}-1)-\sum _{k=-\infty}^{+\infty}(e^{2\pi i n_km_l\alpha}-1)e^{2\pi in_kx_{n_l}}\\
&&\qquad=\sum_{k=-l}^{l}(e^{2\pi i n_km_l\alpha}-1)-\sum _{k=-l}^{l}(e^{2\pi i n_km_l\alpha}-1)e^{2\pi in_kx_{n_l}}\\
&&\qquad=2\sum_{k=1}^{l}\{(\cos\,2\pi n_k m_l\alpha-\cos\,0)-[\cos\,2\pi n_k(m_l\alpha+x_{n_l})-\cos\,2\pi n_k x_{n_l}]\}\\
&&\qquad=8\sum_{k=1}^{l}\sin\,(\pi n_k m_l \alpha)\,\sin\,(\pi n_k x_{n_l})\,\cos\,\pi n_k (m_l\alpha+x_{n_l})\\
&&\qquad=8\sum_{k=1}^{l}\sin\,\Big(\pi \sum_{j=l+1}^{\infty}\frac{n_k m_l}{n_j}\Big)\,\sin\,
(\pi n_k x_{n_l})\,\cos\, \Big(\pi\sum_{j=l+1}^{\infty}\frac{n_k m_l}{n_j}+ \pi n_k x_{n_l}\Big)
\end{eqnarray*}
Now fix $k\in [1,l]$. Since $\frac{n_k}{n_1 n_l}<\sum_{j=l+1}^{\infty}\frac{n_k m_l}{n_j}<\frac{2}{n_1}$
and $x> \sin\, x>\frac{2}{\pi}x,\,x\in (0,\frac{\pi}{2})$,
$$
 \frac{2\pi}{n_1} > \sin\,\Big(\pi \sum_{j=l+1}^{\infty}\frac{n_k m_l}{n_j}\Big)>\frac{2}{\pi}\,\cdot\pi
 \sum_{j=l+1}^{\infty}\frac{n_k m_l}{n_j}>\frac{2n_k}{n_1n_l}.
$$
Similarly we can obtain
$$
\frac{\pi n_k}{n_1n_l}> \sin\,(\pi n_k x_{n_l})>\frac{2}{\pi}\,\cdot \pi n_k x_{n_l}=\frac{2n_k}{n_1n_l},
$$
Moreover, since $\frac{2n_k}{n_1 n_l}<\sum_{j=l+1}^{\infty}\frac{n_k m_l}{n_j}+ n_k x_{n_l}
<\frac{3}{n_1}$ and $1>\cos\,x>-\frac{2}{\pi}x+1,\,x\in (0,\frac{\pi}{2})$ we have
$$
\cos\, \Big(\pi\sum_{j=l+1}^{\infty}\frac{n_k m_l}{n_j}+ \pi n_k x_{n_l}\Big)>1-\frac{6}{n_1}.
$$
Hence
$$
\frac{1}{10}>\frac{16\pi^2}{n_1^2}\Big(1+\frac{l-1}{n_{l-1}^2}\Big)>\sum_{t=0}^{m_l-1}\phi(0+t\alpha)-\sum_{t=0}^{m_l-1}\phi(x_{n_l}+t\alpha)>
\frac{32}{n_1^2}(1-\frac{6}{n_1})>\frac{1}{1000}.
$$
Set $\delta=1/1000$, we immediately have \eqref{cond:*} follows.
%$$d(T^{m_l}(x_{n_l},0),T^{m_l}(0,0))=\max\Bigg\{\frac{1}{n_1n_l},\left|\sum_{t=0}^{m_l-1}
%\phi(x_{n_l}+t\alpha)-\sum_{t=0}^{m_l-1}\phi(0+t\alpha)\right|\Bigg\}>\frac{1}{24}.$$

Now it remains to check the uniform rigidity of $(\T^2,T)$.
Let $\ep>0$ and choose $s>1$ with $\frac{\alpha}{n_s}<\frac{\ep}{28}$ and $\frac{2}{n_s^2}<\ep$. Then by (b) we obtain that
$$\left|\sum_{t=0}^{n_s-1}\phi(x+t\alpha)\right|\le \sum _{k=-\infty}^{+\infty}\left|e^{2\pi i n_kn_s\alpha}-1\right|=2\sum_{k=1}^{+\infty}\left|e^{2\pi i n_kn_s\alpha}-1\right|<2\sum_{k=1}^{+\infty}\frac{14}{n_sn_k}< \frac{28\alpha}{n_s}<\ep$$
for any $x\in \T^1$. Clearly, we also have
$$
n_s\alpha \modone \le \sum_{j=s+1}^\infty\frac{n_s}{n_j} \le \frac{2}{n_s^2}<\ep.
$$
This follows that for every $(x,y)\in \T^2$,
$$d(T^{n_s}(x,y),(x,y))=\max\Big\{d(x+n_s\alpha,x),d(\sum_{t=0}^{n_s-1}\phi(x +t\alpha)+y,y)\Big\}<\ep$$
completing the proof.
\end{proof}

\section{The proof of Claim A in Theorem~\ref{countable}}

Recall that we denote $A=X^{(\alpha)}=\{a_1,a_2,\ldots,a_m\}$, where $\alpha=d(X)$ and further assume that each $a_i \in A$ is a fixed point.
Without loss of generality, we also choose fixed (arbitrarily small) $\ep>0$
such that $B_\ep(a_j)=\ol{B_\ep(a_j)}$ for each $a_j \in A$ and $TB_{\ep}(a_i)\cap B_{\ep}(a_j)=\emptyset$ when $i \neq j$. Denote $B_\ep(A)=\bigcup_{j=1}^mB_\ep(a_j)$, and we divide $X$ into three parts:
\begin{enumerate}[(I)]
	\item $X_1=\{x\in X\colon x\notin B_\ep(A)\}$,
	\item $X_2 =\bigcup_{j=1}^{m}X_{2,j}, \ \mbox{where} \ \ X_{2,j}=\{x\in X: x \in B_{\ep}(a_j) \ \mbox{and}\ \textrm{orb}(x,T) \not\subset B_{\ep}(a_j)\}$,
	\item $X_3 =\bigcup_{j=1}^{m}X_{3,j}, \ \mbox{where} \ \ X_{3,j}=\{x\in X: x \in B_{\ep}(a_j) \ \mbox{and}\ \textrm{orb}(x,T) \subset B_{\ep}(a_j)\}.$
\end{enumerate}

Now we shall prove

\noindent {\bf Claim A: } There exists $n\in\N$ such that $\widetilde{X}_2 =\emptyset$, where $\widetilde{T}=T^n$ and
\begin{enumerate}[(I\')]
	\item $\widetilde{X}_1=\{x\in X\colon x\notin B_\ep(A)\}$,
	\item $\widetilde{X}_2 =\bigcup_{j=1}^{m}\widetilde{X}_{2,j}, \ \mbox{where} \ \ \widetilde{X}_{2,j}=\{x\in X: x \in B_{\ep}(a_j) \ \mbox{and}\ \textrm{orb}(x,\widetilde{T}) \not\subset B_{\ep}(a_j)\}$,
	\item $\widetilde{X}_3 =\bigcup_{j=1}^{m}\widetilde{X}_{3,j}, \ \mbox{where} \ \ \widetilde{X}_{3,j}=\{x\in X: x \in B_{\ep}(a_j) \ \mbox{and}\ \textrm{orb}(x,\widetilde{T}) \subset B_{\ep}(a_j)\}.$
\end{enumerate}

\medskip
\noindent {\bf Step 1}: $d(X)=1$ and $X^{(1)}=A=\{a_1,a_2,\ldots,a_m\}$.

\medskip
In this case $X_2$ is finite, to prove this we just need to show that for each
$a_t \in A$, $X_{2,t}$ is finite. If not, then there exists an $a_t \in A$ such that
$$X_{2,t}=\{x\in X: x \in B_{\ep}(a_t) \ \mbox{and}\ \textrm{orb}(x,T) \not\subset B_{\ep}(a_t)\}$$ is infinite. Since every point in $X$ is a periodic point, by choosing only one point in each orbit, there are $x_i \in X_{2,t}$ for each $i \in \N$ and the orbits of $x_i$'s are pairwise disjoint.

Note that by the definition, if $x_i\in X_{2,t}$ then
$$\textrm{orb}(x_i,T) \cap (X\setminus B_{\ep}(A)) \neq \emptyset$$
and $X\setminus B_{\ep}(A)$ is finite (otherwise we can find another point except ponit in $A$ with degree $1$), a contradiction. This means that $X_{2,t}$ is finite for any $t=1,2,\dots,m $ and hence $X_2$ is finite. Let $n$ be the common period of each point in $X_2$ and denote $\widetilde{T}=T^n$, then each point in $X_2$ is a fixed point and  $\widetilde{X}_2 =\emptyset$, completing this case.

\medskip
\noindent {\bf Step 2}: $d(X)=\beta+1$  and $X^{(\beta+1)}=A$.

\medskip
%Consider two cases: (i) $\beta$ is not a limit ordinal number; (ii) $\beta$ is a limit ordinal number.

%To prove that $\widetilde{X}_2$ is an empty set under some higher iterate, we just need to show that so is $\widetilde{X}_{2,t}$  for each $t=1,2,3,\dots,m$.
In Step 2 we consider two cases, i.e. $\beta$ is not a limit ordinal number and $\beta$ is a limit ordinal number.

\smallskip
\textbf{(i):} $\beta$ is not a limit ordinal number: To give the general idea of the proof we first consider the
following case.

\smallskip
\textbf{(i.1):} Firstly we assume $\beta<\aleph_0$, where $\aleph_0$ is the first limit ordinal number.
 %We want to show $\widetilde{X}_2\cap X^{(\gamma)}$ is empty for all $\gamma\le \beta$ under a finite number of higher iterates.

\smallskip
Notice that $X_2\cap X^{(\beta)}$ is a finite set (otherwise we have a point of derived degree $\beta+1$ outside $B_\ep(A)$),
then $X_{2,t}\cap X^{(\beta)}$ is a finite set for each $t\in \{1,2,\dots,m\}$, say $X_{2,t}\cap X^{(\beta)}=\{y_{t,1},\dots,y_{t,k_t}\}$.
Choose $n_1\in\N$ be the common period of points in $X_2\cap X^{(\beta)}$. Hence each point in $X_2\cap X^{(\beta)}$ under $T^{n_1}$ is a fixed point. We denote $X_2^\beta=\bigcup_{j=1}^{m}X^\beta_{2,j}, \ \mbox{where} \ \ X^\beta_{2,j}=\{x\in X: x \in B_{\ep}(a_j) \ \mbox{and}\ \textrm{orb}(x,T^{n_1}) \not\subset B_{\ep}(a_j)\}$. Therefore $X_2^\beta\cap X^{(\beta)}=\emptyset$.

Now we show $X_2^\beta\cap X^{(\beta-1)}$ is a finite set. Assume on the contrary there exists $a_{t} \in A$ such that $X_{2,t}^\beta \cap X^{(\beta-1)}$ is an infinite set, i.e. there are infinitely many points $x_i\in B_\ep(a_t)$ with derived degree $\beta-1$ such that $\textrm{orb}(x_i,T^{n_1})\not \subset B_\ep(a_t)$ for each $i\in\N$. We may assume that the orbits of $x_i$'s are pairwise disjoint (Since $x_i$ is a periodic point, there are infinitely many disjoint periodic orbits, by taking one point in each orbit we can find infinitely many points $x_i^{'}\in B_\ep(a_t)$ with derived degree $\beta-1$ such that $\textrm{orb}(x_i^{'},T^{n_1})\not \subset B_\ep(a_t)$ for each $i\in\N$).
Since $X_2^\beta\cap X^{(\beta)}=\emptyset$, then $X^\beta_{2,t}\cap X^{(\beta)}=\emptyset$. For each $y \in B_{\ep}(a_t)\cap X^{(\beta)}$, there exists
$\delta_y>0$ such that
\begin{equation}
B_{\delta_y}(y)\subset B_\ep(a_t)\;\ \mbox{and}\;\ T^{n_1}(B_{\delta_y}(y))\subset B_\ep(a_t).\label{cond:1}
\end{equation}
Note that $\bigcup_{y \in B_{\ep}(a_t)\cap X^{(\beta)}}B_{\delta_y}(y) \supset (B_{\ep}(a_t)\cap X^{(\beta)})$ and $B_{\ep}(a_t)\cap X^{(\beta)}$ is compact.
Then there are $\{y_{t,1},y_{t,2},\dots, y_{t,l_t}\} \subset B_{\ep}(a_t)\cap X^{(\beta)}$ and $\delta_{t,j}>0$ such that
\begin{equation}
B_{\delta_{t,j}}(y_{t,j})\subset B_\ep(a_t)\;\ \mbox{and}\; \ T^{n_1}(B_{\delta_{t,j}}(y_{t,j}))\subset B_\ep(a_t) \label{cond:2}
\end{equation}
for $1\le j\le l_t$. And $\bigcup_{j=1}^{l_t}B_{\delta_{t,j}}(y_{t,j}) \supset (B_{\ep}(a_t)\cap X^{(\beta)})$ .

It is not hard to see $(B_\ep(a_t)\setminus (\bigcup_{j=1}^{l_t}B_{\delta_{t,j}}(y_{t,j})))\cap X^{(\beta-1)}$ is finite, so we may assume that
$\textrm{orb}(x_i,T^{n_1})\cap  (B_\ep(a_t)\setminus (\bigcup_{j=1}^{l_t}B_{\delta_{t,j}}(y_{t,j})))\neq\emptyset$ for each $i\in \N$. By \eqref{cond:2} there are $w_i\in \textrm{orb}(x_i,T^{n_1})\cap  (B_\ep(a_t)\setminus (\bigcup_{j=1}^{l_t}B_{\delta_{t,j}}(y_{t,j})))$ and $m_i\in\N$ such that
\begin{equation*}
(T^{n_1})^{m_i} w_i\in (B_\ep(a_t)\setminus (B_\delta(a_t) \cup \bigcup_{j=1}^{l_t} B_{\delta_j}(y_j)))\cap X^{(\beta-1)}\;\ \mbox{and}\;\ (T^{n_1})^{m_i+1} w_i\notin B_\ep(x_0)
\end{equation*}
for each $i\in \N$. This is impossible, since
$(B_\ep(a_t)\setminus (\bigcup_{j=1}^{l_t}B_{\delta_{t,j}}(y_{t,j})))\cap X^{(\beta-1)}$ is finite.
So $X^\beta_{2,t}\cap X^{(\beta-1)}$ is finite and hence $X_2^\beta\cap X^{(\beta-1)}$ is finite. Choose $n_2\in\N$ with $n_1|n_2$ be the common period of points in $X_2^\beta\cap X^{(\beta-1)}$, and then each point in $X_2^\beta\cap X^{(\beta-1)}$ under $T^{n_2}$ is a fixed point. We denote $X_2^{\beta-1}=\bigcup_{j=1}^{m} X^{\beta-1}_{2,j}, \ \mbox{where} \ \ X^{\beta-1}_{2,j}=\{x\in X: x \in B_{\ep}(a_j) \ \mbox{and}\ \textrm{orb}(x,T^{n_2}) \not\subset B_{\ep}(a_j)\}$. Therefore $X_{2}^{\beta-1}\cap X^{(\beta-1)}=\emptyset$ and $X_{2}^{\beta-1}\cap X^{(\beta)}=\emptyset$.

Now we check $X_2^{\beta-1}\cap X^{(\beta-2)}$. We just need to show that $X_{2,t}^{\beta-1}\cap X^{(\beta-2)}$ is finite for
each $t=1,2,\dots,m$.

Since $B_{\ep}(a_t)\cap X^{(\beta-1)}$ is compact, there are $z_{t,u} \in B_{\ep}(a_t) \cap X^{(\beta-1)}$ and $\tilde{\delta}_{t,u}>0$ such that
\begin{equation}
B_{\tilde{\delta}_{t,u}}(z_{t,u})\subset B_\ep(a_t)\;\ \mbox{and}\;\ T^{n_2}(B_{\tilde{\delta}_{t,u}}(z_{t,u}))\subset B_\ep(a_t) \label{cond:3}
\end{equation}
for $1\le u\le s$ and $\bigcup_{u=1}^{s}B_{\tilde{\delta}_{t,u}}(z_{t,u})\supset (B_{\ep}(a_t)\cap X^{(\beta-1)}).$
Then
$$(B_\ep(a_t)\setminus (\bigcup_{u=1}^{s}B_{\tilde{\delta}_{t,u}}(z_{t,u})))\cap X^{(\beta-2)}$$
 is finite. Similar as the above method we can get $X_{2,t}^{\beta-1}\cap X^{(\beta-2)}$ is finite and hence $X_2^{\beta-1}\cap X^{(\beta-2)}$ is finite. Choose $n_3\in\N$ with $n_2|n_3$ be the common period of points in $X_2^{\beta-1}\cap X^{(\beta-2)}$. Thus each point in $X_2^{\beta-1}\cap X^{(\beta-2)}$ under $T^{n_3}$ is a fixed point. We  denote $X_2^{\beta-2}=\bigcup_{j=1}^{m}X_{2,j}^{\beta-2}, \ \mbox{where} \ \ X_{2,j}^{\beta-2}=\{x\in X: x \in B_{\ep}(a_j) \ \mbox{and}\ \textrm{orb}(x,T^{n_3}) \not\subset B_{\ep}(a_j)\}$. Therefore $X_{2}^{\beta-2}\cap X^{(\beta)}=\emptyset$, $X_{2}^{\beta-2}\cap X^{(\beta-1)}=\emptyset$ and $X_{2}^{\beta-2}\cap X^{(\beta-2)}=\emptyset$.

Repeating the above arguments a finite number of times, we can find an $n\in\N$ and let $\widetilde{T}=T^n$ then
$\widetilde{X}_2 \cap X^{(\gamma)}= \emptyset$ for all $\gamma\le \beta$.
Hence $\widetilde{X}_2=\emptyset$.

\smallskip
\textbf{(i.2):} Now assume $\beta=\alpha_1+\tilde{n}$ for some limit ordinal number $\alpha_1$ and $\tilde{n} \in \N$.

\smallskip
Repeat the same discussion as above we have $n(\alpha_1) \in \N$ such that $X_2^{\alpha_1} \cap X^{(\xi)}=\emptyset$ for any $\alpha_1 \le \xi \le \alpha_1+\tilde{n}$,
here $X_2^{\alpha_1} =\bigcup_{j=1}^{m}X_{2,j}^{\alpha_1}$ and $X_{2,j}^{\alpha_1}=\{x\in X: x \in B_{\ep}(a_j) \ \mbox{and}\ \\
\textrm{orb}(x,T^{n(\alpha_1)}) \not\subset B_{\ep}(a_j)\}$.

Note that $B_{\ep}(a_t)\cap X^{(\alpha_1)}$ is compact, there exist $\{y_1,y_2,\dots,y_k\} \in B_{\ep}(a_t) \cap X^{(\alpha_1)}$ and $\delta_1,\dots, \delta_k>0$ such that
$\bigcup_{j=1}^{k} B_{\delta_j}(y_j) \supset B_{\ep}(a_t)\cap X^{(\alpha_1)}$ and
\begin{equation}
B_{\delta_j}(y_j) \subset B_{\ep}(a_t) \ \mbox{and} \ \ T^{n(\alpha_1)} (B_{\delta_j}(y_j)) \subset B_{\ep}(a_t) \label{cond:4}
\end{equation}
for $1\le j\le k$.

Note that there exists $\beta_1<\alpha_1$ such that
\begin{equation}
(B_{\ep}(a_t)\setminus \bigcup_{j=1}^{k}B_{\delta_j}(y_j))\cap X^{(\beta_1)}\neq \emptyset  \label{cond:5}
\end{equation}
is finite and
\begin{equation}
(B_{\ep}(a_t)\setminus \bigcup_{j=1}^{k}B_{\delta_j}(y_j))\cap X^{(\gamma)}= \emptyset \label{cond:6}
\end{equation}
 for any $\beta_1<\gamma<\alpha_1$.

%\begin{equation*}
%B_{\sigma}(B_\ep(x_0)\cap X^{(\alpha_1)})\subset B_\delta(x_0) \bigcup \bigcup_{j=1}^l B_{\delta_j}(y_j)\bigcup\bigcup_{t=1}^s B_{\delta_t}(z_t) \eqno{(4)}
%\end{equation*}
%and
%\begin{equation*}
%B_{\sigma}((X\setminus B_\ep(x_0))\cap X^{(\alpha_1)})\cap B_\ep(x_0)=\emptyset. \eqno{(5)}
%\end{equation*}
%By the fact that $\bigcap_{\gamma<\alpha_1} X^{(\gamma)}= X^{(\alpha_1)}$ and label (5), for this $\sigma>0$ there is a non-limit ordinal number $\beta_1<\alpha_1$ such that
%\begin{equation*}
%B_\ep(x_0)\cap X^{(\beta_1)}\subset B_{\sigma}(B_\ep(x_0)\cap X^{(\alpha_1)}). \eqno{(6)}
%\end{equation*}

Now for any $\beta_1< \gamma<\alpha_1$ and any $x\in B_\ep(a_t)\cap X^{(\gamma)}$, using \eqref{cond:1}, \eqref{cond:2}, \eqref{cond:4} and \eqref{cond:6}  we have $\textrm{orb}(x,T^{n(\alpha_1)})\subset B_\ep(a_t)$, so $X_{2,t}^{\alpha_1}\cap X^{(\gamma)}=\emptyset$,
and $X_{2,t}^{\alpha_1}\cap X^{(\beta_1)}$ is finite. Hence $X_{2}^{\alpha_1}\cap X^{(\gamma)}=\emptyset$,
and $X_{2}^{\alpha_1}\cap X^{(\beta_1)}$ is finite. We choose $n(\beta_1) \in \N$ with $n(\alpha_1)|n(\beta_1)$ be the common period of points in $X_{2}^{\alpha_1}\cap X^{(\beta_1)}$,
then each point in $X_{2}^{\alpha_1}\cap X^{(\beta_1)}$ is a fixed point under $T^{n(\beta_1)}$.  We denote $X_2^{\beta_1}=\bigcup_{j=1}^{m}X_{2,j}^{\beta_1}, \ \mbox{where} \ \ X_{2,j}^{\beta_1}=\{x\in X: x \in B_{\ep}(a_j) \ \mbox{and}\ \textrm{orb}(x,T^{n(\beta_1)}) \not\subset B_{\ep}(a_j)\}$.
Therefore $X_{2}^{\beta_1}\cap X^{(\beta_1)}=\emptyset$ and $X_{2}^{\beta_1}\cap X^{(\gamma)}=\emptyset$  for
%$\beta_1< \gamma < \alpha_1$ and $\alpha_1 \le \gamma \le \alpha_1+\tilde{n}$,
%for simplicity we denote it as
$\beta_1< \gamma \le \alpha_1+\tilde{n}$.

If $\beta_1=\alpha_2+\tilde{n}_2$, where $\alpha_2$ is a limit ordinal and $\tilde{n}_2\in\Z_+$.
We repeat the above process to get some $\beta_2=\alpha_3+\tilde{n}_3<\alpha_2$ such that there exists $n(\beta_2)$ with $n(\beta_1)|n(\beta_2)$ and $X_2^{\beta_2}\cap X^{(\gamma)}=\emptyset$
for $\beta_2 \le \gamma \le \alpha_1+\tilde{n}$ (Here $X_2^{\beta_2}=\bigcup_{j=1}^{m}X_{2,j}^{\beta_2}, \ \mbox{and} \ \ X_{2,j}^{\beta_2}=\{x\in X: x \in B_{\ep}(a_j) \ \mbox{and}\ \textrm{orb}(x,T^{n(\beta_2)}) \not\subset B_{\ep}(a_j)\}$).
Then we continue in this manner to discuss $\alpha_3$. Since $d(X)$ is a countable ordinal number,
% we can assume that $d(X)$ is of the form $c_m {\aleph_0}^m+c_{m-1}{\aleph_0}^{m-1}+\dots+c_1\aleph_0+c_0$,
by finitely many repetitions of this procedure we may conclude that
there must be some $\beta_n<\aleph_0$ and $n(\beta_n) \in \N$
%with $n(\beta_{n-1})|n(\beta_n)$
such that $X_{2}^{\beta_n}\cap X^{(\gamma)}=\emptyset$ for any $\beta_n \le \gamma \le \alpha_1+\tilde{n}$. And repeat the argument in~(i-1)
we can find $n \in \N$ and let $\widetilde{T}=T^n$ such that $\widetilde{X}_2 \cap X^{(\gamma)}=\emptyset$ for any $\gamma \le \alpha_1+\tilde{n}$.
Therefore we have $\widetilde{X}_2=\emptyset$.

%we get
%$\bigcup_{\gamma<\beta_n} (X_{2,t}\cap X^{(\gamma)})$ is finite. That is $X_{2,t}=\bigcup_{\gamma\le\beta}(X_2\cap X^{(\gamma)})$ is finite and we end this case.
%{\color{red}Please check if the above is ok. In each step we have to change $T$ by $T^r$ for some $r>0$. Each such change also changes $X_2$ and $X_1$.
%If we do it finitely many times then it is fine. But if we need infinitely many steps it is not completely clear if all will work properly. Maybe we need some %additional explanation here. We clearly cannot change $T$ at start if we have infinitely many intermediate steps between two ordinals. Please check it %carefully once again.}

\smallskip
\textbf{(ii):} $\beta$ is a limit ordinal number:

\smallskip
%We just need to check for each $X_{2,t}$.
The proof is similar to the case {\bf (i.2)}.

\medskip
\noindent {\bf Step 3}: $d(X)=\alpha$ and $X^{(\alpha)}= A$, where $\alpha$ is a limit ordinal number and points in $A$ are fixed points.
The proof is similar to the case {\bf (i.2)}.

Now we have proved Claim A.

\section*{Acknowledgments}
We thank J. Auslander for sending us a useful references, and E. Glasner, W.~Huang and S.~Shao for very useful suggestions.
% I also thank the anonymous referee for his helpful comments that improved the manuscript.

Research of
J. Li was supported by NNSF of China (grant no. 11371339), P.~Oprocha was supported by the European Regional Development
Fund in the IT4Innovations Center of Excellence project (CZ.1.05/1.1.00/02.0070), X.~Ye was supported by NNSF of China (grant no. 11371339 and 11431012) and R. Zhang
was supported by NNSF of China (grant no. 11001071 and 11171320).

%%%%%%%%%%%%%%%%%%%%%%%%%%%%%%%%%%%%%%%%%%%%%%%%%%%%%%%%%%%%%%%%%%%%%%%%%%%%%%%%%%%%%%%%%%%%%%%%%%%%%%
\providecommand{\bysame}{\leavevmode\hbox to3em{\hrulefill}\thinspace}
\providecommand{\MR}{\relax\ifhmode\unskip\space\fi MR }
\providecommand{\MRhref}[2]{%
  \href{http://www.ams.org/mathscinet-getitem?mr=#1}{#2}
}
\providecommand{\href}[2]{#2}

\end{document}